\documentclass{amsart}
\usepackage[mathscr]{eucal}
\usepackage{amsmath,amssymb,hyperref}
\usepackage{amstext}
\usepackage{graphics}
\usepackage{xcolor}
\usepackage{mathrsfs}
\usepackage{epstopdf}

\usepackage{tikz}
\usepackage[toc,page]{appendix}
\usetikzlibrary{matrix}

\usepackage[all]{xy}

\usepackage[T1]{fontenc}
\usepackage[utf8]{inputenc}
\usepackage[french,english]{babel}

\newtheorem{thm}{Theorem}[section]
\newtheorem{THM}{Theorem}

\newtheorem{cor}[thm]{Corollary}
\newtheorem{prop}[thm]{Proposition}

\newtheorem{lemma}[thm]{Lemma}

\theoremstyle{definition}

\newtheorem{definition}[thm]{Definition}
\newtheorem{remark}[thm]{Remark}
\newtheorem{example}[thm]{Example}
\newtheorem{question}[thm]{Question}

\DeclareMathOperator{\Pic}{Pic}
\DeclareMathOperator{\Pict}{Pic^{\tau}}

\DeclareMathOperator{\sing}{Sing}

\DeclareMathOperator{\Aff}{Aff}

\DeclareMathOperator{\Alb}{Alb}

\DeclareMathOperator{\alb}{alb}

\DeclareMathOperator{\utype}{utype}

\DeclareMathOperator{\ord}{ord}

\DeclareMathOperator{\Spec}{Spec}

\def\hot{\mathrm{h.o.t.}}

\def\Diff{\mathrm{Diff}(\mathbb C,0)}
\def\Difffor{\widehat{\mathrm{Diff}}(\mathbb C,0)}

\def\C{\mathbb C}

\def\Q{\mathbb Q}

\def\F{\mathcal F}

\def\X0{X^{\circ}}

\def\Y0{Y^{\circ}}

\input xy
\xyoption{all}

\addtocounter{section}{0}             % Start with section 1
\numberwithin{equation}{section}       % Number formulas within sections

\begin{document}

\title[Compact Leaves]
{Compact leaves of codimension one holomorphic foliations on projective manifolds}
%Titre en français :
%Feuilles compactes des feuilletages holomorphes de codimension 1 sur les variétés projectives

\author[B. Claudon]{Beno\^{\i}t CLAUDON}
\address{IECL, Universit\'e de Lorraine, BP 70239, 54506 Vand{\oe}uvre-l\`es-Nancy Cedex, France}
\email{benoit.claudon@univ-lorraine.fr}

\author[F. Loray]{Frank Loray}
\address{Univ Rennes, CNRS, IRMAR - UMR 6625, F-35000 Rennes, France}
\email{frank.loray@univ-rennes1.fr}

\author[J.V. Pereira]{Jorge Vit\'{o}rio PEREIRA}
\address{IMPA, Estrada Dona Castorina, 110, Horto, Rio de Janeiro,
Brasil}
\email{jvp@impa.br}

\author[F. Touzet]{Fr\'ed\'eric Touzet}
\address{Univ Rennes, CNRS, IRMAR - UMR 6625, F-35000 Rennes, France}
\email{frederic.touzet@univ-rennes1.fr}

%Mots-clés en français
%Feuilletages holomorphes de codimension 1 ; Feuilles compactes ; Représentations d'holonomie ; Théorie de Ueda ; Feuilletages transversalement affines.

\thanks{This work was initiated during a research stay of the first author at IMPA. He would like to warmly thank the institute for its hospitality and excellent working conditions. The third author is supported by CNPq and FAPERJ,
and BC, FL and FT benefit from support of CNRS and ANR-16-CE40-0008 project ``Foliage''}

\begin{abstract}
This article studies  codimension one foliations on projective manifolds having a compact leaf
(free of singularities). It explores the interplay between Ueda theory (order of flatness of the normal bundle) and  the holonomy representation (dynamics of the foliation in the transverse direction). We address in particular the following problems: existence of foliation having as a leaf a given hypersurface with topologically torsion normal bundle,
global structure of foliations having a compact leaf whose holonomy is abelian (resp. solvable), and factorization results.
\end{abstract}

\maketitle

\selectlanguage{french}
\begin{center}
\textbf{FEUILLES COMPACTES DES FEUILLETAGES HOLOMORPHES\\
\vspace*{3pt}DE CODIMENSION 1 SUR LES VARI\'ET\'ES PROJECTIVES}
\end{center}
\begin{abstract}
Cet article étudie les feuilletages de codimension 1 sur les variétés projectives admettant une feuille compacte (ne rencontrant pas le lieu singulier du feuilletage). Les interactions entre la théorie de Ueda (ordre de platitude du fibré normal de la feuille) et la représentation d'holonomie (dynamique du feuilletage dans la direction transverse) sont explorées. Nous envisageons en particulier les problématiques suivantes : existence de feuilletages admettant pour feuille une hypersurface donnée possédant un fibré normal topologiquement de torsion, étude de la structure globale des feuilletages ayant une feuille compacte d'holonomie abélienne (resp. résoluble) et résultats de factorisations.
\end{abstract}

\setcounter{tocdepth}{1}
\sloppy
\tableofcontents

\section{Introduction}

 Let $X$ be a complex manifold and $\mathcal F$ a  codimension one singular holomorphic foliation on it.
If $Y$ is a compact leaf of $\mathcal F$ (the foliation is regular along $Y$) then the topology/dynamics
of $\mathcal F$ near $Y$ is determined by the holonomy  of $Y$. Given a base point $p \in Y$ and
a germ of  transversal $(\Sigma,p) \simeq (\mathbb C,0)$ to $Y$ at $p$, we can lift paths on $Y$ to nearby leaves in order to obtain
the holonomy representation of $\mathcal F$ along $Y$
\[
 \rho : \pi_1(Y,p) \longrightarrow \Diff \, .
\]

 The purpose of this article is  to investigate which
representations as above can occur as the holonomy representation
of a compact leaf of a codimension one foliation on a compact K\"{a}hler
manifold, and how they influence the geometry of the foliation $\mathcal F$.

\subsection{Previous results}\label{sec:previousresult}
To the best of our knowledge the  first work to  explicitly study compact leaves of foliations on compact
 manifolds is due to Sad \cite{Sad}.
One of his main results,  \cite[Theorem 1]{Sad}, is stated below.

\begin{thm}\label{T:Sad}
Let $C_0\subset \mathbb P^2$ be a smooth curve of degree $d\ge 3$ and $P$ be a set of $d^2$  very general points
on $C_0$. If $S$ is the blow-up of $\mathbb P^2$ along $P$  and $C$ is the strict transform of $C_0$ in $S$
then $C$ is not a compact leaf of any foliation on $S$.
\end{thm}

This result answered negatively a question of Demailly, Peternell and Schneider \cite[Example 2]{DPS96} about the existence
of foliations having a compact leaf on the blow-up of $\mathbb P^2$ along $9$ very general points. Sad (loc. cit.) obtained other results
for the blow-up of smooth cubics which were later subsumed by the following result of Brunella  build
upon
the classification of foliations according to their  Kodaira dimension, see \cite[Chapter 9, Corollary 2]{Brunella}.

\begin{thm}
Let $\mathcal F$ be a foliation on a projective surface $S$ and suppose $\mathcal F$ admits an elliptic curve $E$
as a compact leaf. Then, either $E$ is a (multiple) fiber of an elliptic fibration or, up to ramified coverings and
birational maps, $E$ is a section of a $\mathbb P^1$-bundle. In the former case, $\mathcal F$ is the elliptic fibration
itself, or is turbulent with respect to it; and in the latter case, $\mathcal F$ is a Riccati foliation.
\end{thm}

Bott's partial connection induces a flat connection on the normal bundle of compact leaves of foliations. Hence, our
subject is also closely related to the study of smooth divisors with numerically trivial normal bundle on
projective or compact K\"{a}hler manifolds. Serre constructed examples of curves on ruled surfaces with trivial normal bundle,
having  Stein complement but without non constant regular (algebraic) function \cite[Chapter VI, Example 3.2]{Hartshorne}.
Precisely, given an elliptic curve $C$, consider the unique unsplit extension $0\to \mathcal O_C\to E\to \mathcal O_C\to 0$;
on the total space $X=\mathbb P(E)$, the curve $Y\simeq C$ defined by the embedding $\mathcal O_C\to E$
provides us with such an example.
Motivated
by these examples, Hartshorne asked if the complement of a curve $C$ on a projective surface $S$ is Stein whenever $C^2 \ge 0$ and $S-C$
contains no complete curves \cite[Chapter VI, Problem 3.4]{Hartshorne}. This question was answered negatively by Ogus,
who exhibits examples
of  rational surfaces \cite[Section 4]{Ogus} containing elliptic curves with numerically trivial normal bundle, without complete curves in its
complement, and which are not Stein. Both  Ogus and Serre examples carry global foliations smooth along the  curves. A
variation on Ogus example is presented as   Example \ref{exampleOgusFred} below.

Ueda carried on the study of smooth curves on surfaces with numerically trivial normal bundle, looking for obstructions
to the existence of certain kind of foliations smooth along the curves. The {Ueda type} $k\in\{1,2,3,\ldots,\infty\}$ of $Y$ ($\utype (Y)$ for short) is, roughly speaking, defined by the first infinitesimal neighborhood of $Y$ for which the flat unitary foliation of $NY$
does not extend as a foliation with linearizable holonomy. See Section \ref{S:utype} for a precise definition.
There, we also review some of the results of Ueda and Neeman following \cite{Ueda} and \cite{Neeman}.
%One of the remarkable result of the theory is \cite[Theorem 6.12]{Neeman}

%\begin{thm} Let $X$ be a projective surface and $Y\subset X$ be an embedded elliptic curve with finite Ueda type $k<\infty$.
%Then the pair $(X,Y)$ is obtained from Serre example $(X_0,Y_0)$ above by blowing-up points outside of $Y_0$.
%\end{thm}

%We note that the assumption $k<\infty$ is fulfilled whenever
%the normal bundle $NY$ is torsion but $Y$ is not a fiber (multiple or not) of a fibration $f:X\to C$.

\subsection{Existence} Although most of our results deal with the global setting, where $X$ is assumed
to be projective or compact K\"ahler, we first prove the existence of a formal foliation in the semi-local setting.

\begin{THM}\label{THM:formalfoliation}
Let $Y\subset X$ a smooth curve embedded in a germ of surface $X$ such that $Y^2=0$. If one denotes by $Y(\infty)$ the formal completion of $X$ along $Y$, then
$Y$ is a leaf of a (formal) regular foliation $\hat{\F}$ on $Y(\infty)$.
\end{THM}

Turning to the global setting, our second result concerns the existence of foliations smooth along hypersurfaces with topologically torsion (or flat) normal bundle.

\begin{THM}\label{THM:Existence}
Let $X$ be a compact K\"{a}hler manifold and $Y$ be a smooth hypersurface in $X$
with normal bundle $NY$. Assume there exists an integer $k$ and a flat line-bundle $\mathcal L$
on $X$ such that $\mathcal L_{|Y}= NY^{\otimes k}$. If $\utype(Y)\ge k$, then there exists a rational $1$-form $\Omega$ with coefficients
in $\mathcal L^*$ and polar divisor equal to $(k+1)Y$. Moreover, if $\nabla$ is the unitary flat connection
on $\mathcal L^*$, then $\nabla(\Omega)=0$, and $\Omega$ defines a transversely affine foliation on $X$ which admits
$Y$ as a compact leaf.
\end{THM}

We refer to \cite{Gaeljvp} for the notion of transversely affine foliation.

When $\utype (Y)> k$, where $k$ is the analytical order of $NY$, this result is due to Neeman \cite[Theorem 5.1, p. 109]{Neeman}, where it is proved that $Y$ fits into a fibration.
\begin{thm}\label{thm:Neeman}
Let $X$ be a compact K\"{a}hler manifold and $Y$ be a smooth hypersurface in $X$
with torsion normal bundle of order $k$.
Assume that the Ueda type of $Y$ satisfies $\utype(Y)>k$, then $\utype(Y)=\infty$ and $kY$ is a fiber of a fibration on $X$.
\end{thm}

The reader interested in other fibration existence criteria may consult \cite{Dingo, jvpJAG, Totaro}.
Remark that, when $NY$ is {analytically trivial},  Theorem \ref{THM:Existence} implies the existence of global foliations on $X$ for which $Y$ is a compact leaf. This fact is in contrast with Theorem \ref{T:Sad}.  Sufficient and necessary conditions for the existence of a global foliation on compact K\"ahler manifold having a given compact leaf remain elusive. For instance, we are not aware of smooth hypersurfaces with torsion normal bundle and Ueda type strictly smaller than the order of the normal bundle which are
not compact leaves. For more on this matter see the discussion at Section \ref{S:QE}.

\subsection{Abelian holonomy}
For arbitrary representation $\rho : \pi_1(Y,p) \longrightarrow \Diff$
there exists a complex manifold
$U$ containing $Y$ as a hypersurface, and a foliation $\mathcal F$ on $U$ leaving $Y$ invariant and with holonomy representation
along $Y$ given by $\rho$, see \cite{Ily}.
If we start with an abelian representation $\rho$, it is well-known (see Section \ref{S:interpretation}) that there exists a formal meromorphic closed $1$-form in
$Y(\infty)$, the completion of $U$  along  $Y$, defining $\mathcal F_{|Y(\infty)}$. In many cases there do not exist (convergent) meromorphic closed $1$-forms
defining $\mathcal F$ on $U$, even if we restrict $U$.
Our third result says that, for codimension one foliations on projective manifolds, this can only
happen when the holonomy is linearizable.

\begin{THM}\label{THM:Abelian}
Let $X$ be a projective manifold and $Y$ be a smooth hypersurface in $X$.
Assume there exists a foliation $\mathcal F$ on $X$ admitting $Y$ as a compact leaf.
If the  holonomy of $\mathcal F$ along
$Y$ is  abelian,  then at least one of the following assertions holds true.
\begin{enumerate}
\item There exists a projective manifold $Z$ and generically finite morphism $\pi : Z \to X$
such that $\pi^* \mathcal F$ is defined by a closed rational $1$-form.
\item The holonomy of $\mathcal F$ along $Y$ is formally linearizable, $NY$ has infinite order and $\utype(Y)=\infty$.
\end{enumerate}
\end{THM}

\begin{question}
Let $\F$ be a foliation on a projective manifold fulfilling the hypothesis of Theorem \ref{THM:Abelian}. Does $\F$ automatically satisfy item (1)?
\end{question}

We can exclude the existence of  hypothetical counter-examples under more demanding, but quite natural hypothesis, see Section \ref{S:FormallyLinearizableHolonomy}.
We also obtain similar results (under additional restrictions) if the holonomy of $\F$ along $Y$ is only supposed to be solvable in Section \ref{S:solvablehol}.

\subsection{Factorization}
If $Y$ is a genus $g>1$ curve, then there is no difficulty to construct non-abelian representations of $\pi_1(Y)$ in $\Diff$.
If $Y$ is a higher dimensional manifold and $\rho : \pi_1(Y) \to \Difffor$ is a non virtually abelian representation then we are able to prove that the representation factors over a curve  in the following sense (see   Section \ref{S:Factor} for a more detailed discussion).
\begin{THM}\label{THM:Holfactorization} Let $Y$ be a compact K\"ahler manifold and $\rho : \pi_1(Y) \to \Difffor$ a morphism. Suppose $G=\mbox{Im}\ \rho$ is not virtually abelian, then its center $Z(G)$ is necessarily a finite subgroup and the induced representation ${\rho}': \pi_1(Y) \to G/Z(G)$ factors through an orbicurve.\end{THM}

Theorem \ref{THM:Holfactorization} admits a foliated counterpart.

\begin{THM}\label{THM:Factorization}
Let $X$ be a compact K\"{a}hler manifold and $Y$ be a smooth hypersurface in $X$.
Assume there exists a foliation $\mathcal F$ on $X$ admitting $Y$ as a compact leaf whose normal bundle has {finite} order $m$. Suppose moreover that $\utype (Y)\geq m$, and that the holonomy representation  $\pi_1(Y) \to \Difffor$ has a non virtually abelian image. Then, there exists a rational map $\pi : X \dashrightarrow  S$ to a projective surface $S$  and a foliation $\mathcal G$ on $S$ such that $\mathcal F = \pi^* \mathcal G$.
\end{THM}

We emphasize that this statement includes the case where $Y$ has {analytically trivial normal bundle} and also the one where  $Y$ is the regular fiber (multiple or not) of a fibration $f:X\to C$. This latter case corresponds to  $\utype (Y)> m$  by virtue of Theorem \ref{thm:Neeman}.

\begin{question}\label{Q:factorization}
Let $Y$ be a smooth hypersurface on a projective manifold $X$. Assume that $Y$ is a compact leaf of a codimension one foliation $\mathcal F$
on $X$ and that its holonomy representation is not virtually abelian. Is it true that $\mathcal F$ is the pull-back
of a foliation on a projective surface?
\end{question}

The present work stemmed from our attempts to answer this last question. Theorem \ref{THM:Factorization} is
the most compelling evidence in its favor we found so far. Moreover, we can provide a complete answer to both questions listed above under quasi-smoothness assumptions (see Section \ref{S:QSmoothFol}).

\subsection{Structure of the article}
Section \ref{S:Ueda} aims at recalling the basics of Ueda theory for compact hypersurfaces with numerically trivial normal bundle and gives an account of results by Ueda and Neeman  used later. The existence of formal foliations is discussed  in Section \ref{S:Existence1} where Theorem \ref{THM:formalfoliation} is proved.
 %\ref{THM:Existence2}
The analogue problem in the global setup is approached in Section \ref{S:Existence2} where we establish Theorem  \ref{THM:Existence}. There we also exhibit a variant of an example by Ogus of a curve contained in a surface whose complement has no compact curves  but is not Stein (Example \ref{exampleOgusFred}). Foliations with abelian holonomy along a compact leaf are studied in Sections \ref{S:subgroups} and \ref{S:abelianhol}. The former section recalls the classifications of solvable subgroups of $\Diff$ and Theorem \ref{THM:Abelian} is proved in the latter. In Section \ref{S:solvablehol}, we try to push one step further the study of foliations according to the complexity of the holonomy and prove some structure results when the holonomy is solvable. Factorization results such as Theorem \ref{THM:Factorization} are proved in Section \ref{S:Factor} and the case of quasi-smooth foliations is treated in Section \ref{S:QSmoothFol}. The paper ends with  an appendix on the extension of
 projective structures.

\section{Neighborhoods of divisors}\label{S:Ueda}

\subsection{Ueda type and Ueda class}\label{S:utype}
 We  recall below the definitions of  Ueda type and the Ueda class of
smooth divisors with topologically torsion normal bundles.
The point of view presented here follows closely  Neeman's exposition of the subject, see \cite{Neeman}. For an attempt to define the Ueda class in the higher codimension case, we refer the reader to the recent preprint \cite{Koi}

Let $Y$ be a smooth irreducible compact hypersurface on a complex manifold $U$. Assume that the normal bundle $N Y$ is topologically torsion
and that $Y$ and $U$ have the same homotopy type.  Let us assume that $N Y$ carries a flat unitary connection (this condition is automatically fulfilled if $Y$ is K\"{a}hler). Associated to it we get
a representation $\rho_Y : \pi_1(Y) \to S^1\subset \mathbb C^*$. Since $U$ and $Y$ have the same homotopy type, then we have a line bundle on $U$ endowed
with a flat unitary connection which we will denote by $(\widetilde{N_UY}, \nabla_U)$. The line bundle $\mathcal O_U(Y)$ is another extension of $N Y$
to the complex manifold $U$. Notice that it comes equipped with a flat logarithmic connection with trivial monodromy locally defined as $\nabla_Y = d - d \log f$
where $f$ is a local equation of $Y$.  Let us call $$(\mathcal U, \nabla)= (\mathcal O_U(Y),\nabla_Y) \otimes (\widetilde{N_UY},  \nabla_U)^*$$ the \textit{Ueda connection} and the underlying line bundle,
the \textit{Ueda line bundle}, which are both defined over $U$.

Let $I\subset \mathcal O_U$ be the ideal sheaf  defining $Y$.
We will denote  the $k$-th infinitesimal neighborhood of $Y$ in $U$ by $Y(k)$, \emph{i.e.}, $Y(k) = \Spec ( \mathcal O_U / I^{k+1} )$.
The \textit{Ueda type}  of $Y$ ($\utype (Y)$) is   defined as follows.
If $\mathcal U_{|Y(\ell)} \simeq \mathcal O_{Y(\ell)}$ for every $\ell < k$ and $\mathcal U_{|Y(k)} \not\simeq \mathcal O_{Y(k)}$, then
the  $\utype(Y)=k$. If  $\mathcal U_{|Y(\ell)} \simeq \mathcal O_{Y(\ell)}$ for every non-negative integer $\ell$
then $\utype(Y)=\infty$.

If $\utype(Y) = k < \infty$ then the {\it Ueda class} of $Y$ is defined as the element in $H^1(Y, I^{k}/I^{k+1})$ mapped  to $\mathcal U_{|Y(k)} \in \Pic(Y(k))$
through the truncated exponential sequence
\[
0\to H^1(Y, I^{k}/I^{k+1}) \to \Pic (Y(k)) \to \Pic(Y(k-1)) \, .
\]

In concrete terms and following Ueda's original definition \cite{Ueda}, $Y$ has Ueda type $\utype(Y)\ge k$ if, and only if
there exists a covering $U_i$ of a neighborhood of $Y$, and submersions $y_i \in \mathcal O(U_i)$ vanishing along $Y$ such that whenever $U_i \cap U_j \neq \emptyset$,
\begin{equation}\label{TransitionUedaCoordinates}
  \lambda_{ij}y_j-y_i =  a_{ij} y_i^{k+1}
\end{equation}
 for suitable $\lambda_{ij} \in S^1$ and $a_{ij} \in \mathcal O(U_i\cap U_j)$.
In this case, $\utype(Y)> k$ if, and only if one can write
$$a_{ij}{}_{|Y}=(\lambda_{ij}^{-k}a_j-a_i){}_{|Y}$$
for some collection of functions $a_i\in\mathcal O(U_i\cap Y)$.
If not, then the restriction $\lbrace\ a_{ij}{}_{|Y}\ \rbrace$ defines the Ueda class of $Y$. As observed by Ueda (loc. cit.), this latter is well defined as an element of $H^1(Y,{NY}^{\otimes -k})$  up to multiplication by a non-vanishing constant which comes from the choice of an isomorphism between $\mathcal O_Y$ and ${\mathcal U}_{|Y}$. Finally, the triviality of the line bundle $\mathcal U$ on an Euclidean neighborhood of $Y$ is precisely equivalent to the existence of a collection of transverse coordinates $y_i \in \mathcal O(U_i)$ with
\[
 y_i = \lambda_{ij}y_j
\]
for suitable $\lambda_{ij} \in S^1$; they are obtained from integration of Ueda connection.

\begin{lemma}\label{L:Uedaholonomy}
Let $\mathcal F$ be a codimension one foliation on a complex manifold $U$ having a compact leaf $Y$. If the holonomy of $\mathcal F$ along $Y$ is linearizable up to order $k$ and its  linear part  is unitary then $\utype(Y)\ge k$. Moreover, in this case
\[
N\mathcal F_{|Y(k)} = \mathcal O_{Y(k)}(Y) \otimes \mathcal U^{\otimes k}_{|Y(k)}.
\]
\end{lemma}
\begin{proof}
The hypothesis on the holonomy of $\mathcal F$ along $Y$ allows us to choose a covering $\{ U_i \}$ of a neighborhood of $Y$ in $U$ and first integrals $y_i \in \mathcal O_U(U_i)$ for
$\mathcal F$ such that
\[
y_i = \underbrace{(\lambda_{ij} + a_{ij} y_j^k + \hot )}_{=y_{ij}} y_j
\]
where $\{\lambda_{ij}\}$ is a cocycle with values in $S^1$ defining $\widetilde{N_U Y}$, and $a_{ij}$ are complex constants.
Note that the cocycle $\{ y_{ij} \}$ defines the line-bundle $\mathcal O_U(Y)$. It follows
that $\utype(Y)\ge k$ according to  Equation (\ref{TransitionUedaCoordinates}).

Differentiating the above expression we obtain that
\[
dy_i = (\lambda_{ij} + (k+1) a_{ij} y_j^k + \hot ) dy_j.
\]
The cocycle $\{(\lambda_{ij} + (k+1) a_{ij} y_j^k + \hot )\}$ represents the normal bundle of $\mathcal F$ and, when restricted
to $Y(k)$, coincides with the cocycle $\{ y_{ij}^{k+1} \cdot  \lambda_{ij}^{-k} \}$. Since this last cocycle represents $\mathcal O_U(Y) \otimes \mathcal U^ {\otimes k}$, the lemma follows.
\end{proof}

\begin{remark}\label{rk:CounterExample}
Without assuming  the linear part of the holonomy unitary, the lemma above does not hold true, even if one assumes
that the holonomy is linearizable. Indeed, we will now proceed to construct an example of a surface $U$ containing  an elliptic curve $Y$
with trivial normal bundle, $\utype(Y)=1$ and such that $Y$ is a leaf of a smooth foliation $\mathcal F$
with linear (non unitary) monodromy.

Write  $Y$ as $\mathbb C / \Gamma$, where $\Gamma = \mathbb Z \oplus \mathbb Z \theta$ and choose a representation  $\gamma \mapsto c_{\gamma}$
of $\Gamma$ in $(\mathbb C,+)$.
We want to construct a pair of foliations $\mathcal F$ and $\mathcal G$ on a neighborhood of $Y$
such that the holonomy of $\mathcal F$ along $Y$ is given by  $\rho_{\mathcal F}(\gamma) = \{  y \mapsto \exp(\gamma) y  \}$;
while the holonomy of $\mathcal G$ is given by  $\rho_{\mathcal G}(\gamma)= \{ y \mapsto \frac{y}{1 + c_{\gamma} y} \}$.
For that sake,
on $\mathbb C \times (\mathbb C,0)$ with coordinates $(x,y)$, consider the functions $F=y$   and $G= \exp(x) y$.

We want to construct $\phi_{\gamma}$ such that
\[
F\circ \phi_{\gamma} = \exp(\gamma) F \quad \text{ and } \quad G \circ \phi_{\gamma} = \frac{G}{1 + c_{\gamma} G}.
\]
From the first equation we deduce that
$
\phi_{\gamma}(x,y) = (a_{\gamma}(x,y) , \exp(\gamma) y)
$
and the second equation implies
\[
\exp(a_{\gamma}(x,y) ) \exp(\gamma) y = \frac{ \exp(x) y } { 1 + c_{\gamma} \exp(x) y } .
\]
This last equation determines $a_{\gamma}$, \emph{i.e.}
\[
a_{\gamma}(x,y) = x - \gamma - \log(1+c_{\gamma} \exp(x) y ) \, .
\]
Hence, if we take
\[
U_{\gamma} = \{ (x,y) \in \mathbb C^2 \, ; \, |y| < |\exp(-x)| |c_{\gamma}^{-1}| \}
\]
then we can define $\phi_{\gamma}:U_{\gamma}\to \mathbb C^2$.
If we take the quotient of $U_1 \cap U_{\theta}$ by $\phi_1$ and $\phi_{\theta}$ we obtain
a surface containing $Y$ and having a pair of foliations with the sought properties.
If we choose the representation $\gamma \mapsto c_{\gamma}$  non commensurable with the periods of $Y$ then $\utype (Y)=1$.
\end{remark}

\begin{remark}
Lemma \ref{L:Uedaholonomy} holds true without unitary assumption if the compact leaf $Y$ admits a germ of transverse fibration.
Indeed, given a system of non unitary flat coordinates as in Equation (\ref{TransitionUedaCoordinates}),
the cocycle $\lbrace \lambda_{ij} \rbrace\in H^1(Y,\mathbb C^*)$ is equivalent in $H^1(Y,\mathcal O^*_Y)$ to a unitary cocycle
$\lbrace \mu_{ij} \rbrace \in H^1(Y,\mathbb C^*)$, \emph{i.e.} $\lambda_{ij}u_j=u_i\mu_{ij}$ with $u_i\in\mathcal O^*(U_i\cap U_j\cap Y)$.
The transverse fibration allow to extend the coboundary $u_i$ to the neighborhood of $Y$, \emph{i.e.} $u_i\in\mathcal O^*(U_i\cap U_j)$,
so that in new variables $y_i=u_i\tilde y_i$ we get a unitary Ueda cocycle  (\ref{TransitionUedaCoordinates}) linear up to order $k$.

In the example of Remark \ref{rk:CounterExample},
there is \textit{a posteriori} no  fibration transversal to $Y$, and this makes impossible to change
the linear part of Equation (\ref{TransitionUedaCoordinates}) into unitary ones without perturbing higher order terms  in the cocycle.
\end{remark}

\subsection{Hypersurfaces of infinite type}

Let $\Pic^{\tau}(Y)$ denote the group of line bundles with torsion Chern class, \emph{i.e.} with zero real Chern class.
Let $d: \Pic^{\tau}(Y) \times  \Pic^{\tau}(Y) \to \mathbb R$ be a homogeneous distance on $\Pic^{\tau}(Y)$, \emph{i.e.} if $L, L'$ and $L''$ are
elements of $\Pic^{\tau}(Y)$ then
$d ( L\otimes L', L \otimes L'') = d( L', L'')$  and
$d(L'^* , L''^*)= d(L',L'').$
Let $\mathfrak{E}_0 \subset  \Pic^{\tau}(Y)$ be the subset of torsion line bundles, and $\mathfrak{E}_1 \subset \Pic^{\tau}(Y) - \mathfrak{E}_0$ be the subset defined by the following Diophantine condition: $L \in \mathfrak E_1$ if and only if
there exists real constants  $\alpha,\epsilon>0$ (depending on $L$) such that
\[
d(\mathcal O_Y, L^{\otimes \nu}) \ge \frac{\epsilon}{\nu^{\alpha}}
\]
for every integer $\nu \ge 1$.

By definition, if $\utype(Y)=\infty$, then the restriction of $\mathcal U$ to the completion of $U$ along $Y$ is trivial, \emph{i.e.} $\mathcal U$ is trivial on a formal neighborhood of $Y$. The theorem below  due to Ueda (cf. \cite[Theorem 3]{Ueda} ) gives sufficient conditions to the
triviality of $\mathcal U$ on an Euclidean neighborhood of $Y$ in $U$.  Although Ueda states his result only for curves on surfaces, his proof works as it is to establish the more general  result below.

\begin{thm}\label{T:Uedainfty}
Let $Y$ be a smooth compact  connected K\"ahler hypersurface of a complex manifold $U$ with topologically torsion normal bundle. If $\utype(Y)= \infty$ and $NY \in \mathfrak{E}_0 \cup \mathfrak{E}_1$, then the Ueda line bundle $\mathcal U$ is trivial
on an Euclidean neighborhood of $Y$.
\end{thm}

Given a compact complex manifold $Y$ with non-zero $H^1(Y,\mathbb Q)$, it is possible to construct a neighborhood of $Y$
(with $Y$ a hypersurface) such that $NY$ is topologically trivial, $\utype(Y)= \infty$, but there is no
Euclidean neighborhood in which $\mathcal U$ is trivial. It suffices to take an abelian  representation of $\pi_1(Y)$ to $\Diff$ which is formally linearizable, but has infinitely many periodic
points converging to zero and consider the suspension of this representation, see \cite[\S 37]{Arnold} and \cite[Section 4]{Ueda}. Since these examples have infinitely many pairwise disjoint hypersurfaces with proportional
Chern classes and which do not fit into a fibration, they are not open subsets of compact complex manifolds, see for instance \cite{jvpJAG} and \cite{Totaro}. Up to date there are no known examples of hypersurfaces of compact manifolds with $\utype(Y)=\infty$ but no Euclidean neighborhood over which the Ueda line bundle $\mathcal U$ is trivial.

\subsection{Curves of finite type}\label{ssec:CurvesFiniteType}
Smooth compact curves $C$ on smooth surfaces $S$ (not necessarily compact) having numerically trivial normal bundle and
finite Ueda type present a mixed behavior, combining features of ample divisors and fiber of fibrations. According
to \cite[Section 5]{Hironaka} the transcendence degree of the field of formal meromorphic functions is infinite (like
for fibers of a fibration). On the other hand, the only formal holomorphic functions are the constants (like for ample divisors): indeed, such a function should be constant along the curve, say zero, and would provide
a (non reduced) equation for $C$, contradicting finiteness of Ueda type.

A much more striking similarity with ample divisors is given by the following result of Ueda \cite[Theorem 1]{Ueda}, see also \cite[Proposition 5.3, page 35]{Neeman}.

\begin{thm}\label{T:Ueda}
Let $C$ be a  smooth curve on a smooth  surface $S$. If  $C^2 = 0$ and $\utype(C) < \infty$, then  there exists
a neighborhood $U\subset S$ of $C$ and a strictly plurisubharmonic function $\Phi:U-C \to \mathbb R$ such that
$\lim_{p \to C} \Phi(p) = \infty$. In particular, if $S$ is compact, then $S-C$ is strongly pseudoconvex, and it contains only finitely many
compact curves; these curves can be contracted, and the resulting space of the contraction is a Stein space.
\end{thm}
Actually, Ueda gives precise estimates of the growth of $\Phi$ near $C$. In particular, such curves embedded in complex compact surfaces provide natural examples of \textit{nef} line bundles which do not carry any smooth hermitian metric with semi-positive curvature (see \cite{Ko}).

From the above theorem, it follows that $C$, like an ample divisor, has a fundamental system of strictly pseudoconcave
neighborhoods \cite[Corollary of Theorem 1]{Ueda}. Consequently \cite[Theorem 4]{Andreotti},
the field of germs at $C$ of (convergent) meromorphic functions has transcendence degree bounded by the dimension of the ambient space, like for
an ample divisor.

We will use mainly this pseudoconvexity statement through its combination with \cite[Theorem 3 and Lemma 5]{Ueda2}. This gives rise to the following extension result.
\begin{thm}\label{T:Uedaextension}
Let $C$ be a  smooth curve on a smooth  surface $S$ and let $\mathcal{E}$ be a holomorphic vector bundle on $S\backslash C$. If  $C^2 = 0$ and $\utype(C) < \infty$, then any holomorphic section of $\mathcal{E}$ defined on $U\backslash C$ (where $U$ is an euclidean neighborhood of $C$) has a meromorphic extension to the whole of $S\backslash C$.

Moreover, in the particular case $\mathcal{E}=\Omega^1$, any closed holomorphic one form defined on $U\backslash C$ extends holomorphically to $S\backslash C$.
\end{thm}

\subsection{Hypersurfaces of finite type}
Back to the case of divisors on projective manifolds, or more generally on compact K\"{a}hler manifolds, we
have yet another similarity between numerically trivial divisors with finite Ueda type and ample divisors: the  Lefschetz-like statement below.

\begin{prop}\label{P:Albanese}
Let $X$ be a compact K\"{a}hler manifold and $Y$ a smooth divisor on $X$ with numerically trivial normal bundle.
If $H^1(X,\mathcal O_X) \to H^1(Y,\mathcal O_Y)$ is not  injective, then $Y$ is a multiple fiber of a fibration on
$X$.
In particular, if $\utype(Y)<\infty$, then the restriction morphisms $H^1(X,\mathbb Q) \to H^1(Y,\mathbb Q)$ and $H^1(X,\mathcal O_X) \to H^1(Y,\mathcal O_Y)$
are injective.
\end{prop}
\begin{proof}
If $H^1(X,\mathcal O_X) \to H^1(Y,\mathcal O_Y)$ is not injective, then the morphism between Albanese varieties
$\Alb(Y) \to \Alb(X)$ is not surjective. The composition of the Albanese morphism of $X$ with the quotient map $\Alb(X) \to \Alb(X)/\Alb(Y)$
is a non-constant morphism, which contracts $Y$. Since the normal bundle of $Y$ is numerically trivial, it follows that some multiple of $Y$
is a fiber of a fibration.
For more details see \cite[page 104 and proof of Theorem 5.3 in pages 109-110]{Neeman} and  \cite[proof of Theorem 2.1]{Totaro}.
\end{proof}

A slightly more general version, where we replace smooth divisor by a simple normal crossing divisor also holds true, see
\cite[proof of Theorem 2.3]{croco5}.

\section{Existence of formal foliations}\label{S:Existence1}

In this section we present the proof of Theorem \ref{THM:formalfoliation}.
Its content is not used in the remainder of the article, and the readers interested only on the global aspects
of our study can safely skip it.

\subsection{Notation}
Before going into the details of the proof of Theorem \ref{THM:formalfoliation}, let us introduce the following notations and definitions.

Select  an open covering $\lbrace U_i\rbrace $ of some neighborhood of $Y$. Let $\mathcal V=\lbrace V_i\rbrace$ be the open covering of $Y$ defined by $V_i:=U_i\cap Y$. One can choose $\lbrace U_i\rbrace$ in such a way that $\mathcal V$ is an acyclic covering by disks. Denote by $\widehat{U_i}$ the formal completion of $U_i$ along $V_i$. Let $\{y_i\}\in {\mathcal O}(\widehat{U_i})$ be an {\it admissible coordinates}, which means that $y_i$ is some formal submersion such that $\{{y_i}=0\}=V_i$ and such that
$$t_{ij}y_j -{y_i}= a_{ij}y_i^{2}$$
where the cocycle $\lbrace t_{ij} \rbrace \in Z^1(\mathcal V, {NY}^*)$ is unitary and $\lbrace a_{ij|V_i\cap V_j}\rbrace \in Z^1(\mathcal V, {NY}^{*})$.

\begin{definition}
The coordinates $\lbrace y_i\rbrace$ as above are said to be {foliated} whenever $y_i$ and $y_j$ satisfy
$$ t_{ij}y_j-y_i=\sum_{l=2}^{+\infty} a_{ij}^{(l)}{y_i}^l$$
where the $a_{ij}^{(l)}$'s are {locally constant}.
\end{definition}

\begin{remark}
Note that a system of admissible coordinates is foliated if and only if $dy_i\wedge dy_j=0$. The existence of such coordinates is thus equivalent to the existence of a regular foliation on $Y(\infty)$ leaving $Y$ invariant.
\end{remark}

We will make use of the following classical vanishing result (see for instance \cite[end of \S 1]{Ueda}).

\begin{lemma}\label{vanishingH2}
Let $\Sigma$ be a {non trivial} rank one unitary local system on a smooth
compact and connected  curve $Y$, then
$$H^2(Y, \Sigma)=0.$$
\end{lemma}

Let $\nabla$ the unitary flat connection defined on $N_Y$ and $\Sigma_1$ be the corresponding rank one local system. More generally, for each integer $k$, we will denote by $\Sigma_k$ the local system associated to $\nabla^{\otimes k}$ whose underlying line bundle is then $NY^{\otimes k}$.

Let us fix some positive integer $\nu$.  Let $\lbrace y_i\rbrace$ be a system of admissible coordinates and suppose that one can write
\begin{equation}\label{coefconstants}
t_{ij}{y_j }-y_i= a_{ij}^{(\nu +1)}{y_i}^{\nu + 1}+......+a_{ij}^{(\mu)}{y_i}^{\mu}+a_{ij}^{(\mu +1)}{y_i}^{\mu +1}
\end{equation}
where $\nu<\mu$,   $a_{ij}^{(l)}$ is {locally constant} for $\nu +1\leq l\leq \mu$ and $a_{ij}^{(\mu +1)}\in \mathcal{O}(\widehat{U_i}\cap \widehat{U_j})$.

\begin{definition}
 A system of admissible coordinates $\lbrace y_i \rbrace$ (with respect to $\nu$) which satisfies Equation (\ref{coefconstants}) is said to be a {$\mu$-foliated system of coordinates}.
\end{definition}

This means that we have a foliation on the $\mu^{\text{th}}$ infinitesimal neighborhood.

\begin{remark}
Note that $\lbrace a_{ij}^{(\nu +1)} \rbrace\in Z^ 1(\mathcal V, \Sigma_{-\nu})$. We will denote by $[a_{ij}^{(\nu +1)}]$ the corresponding class in $H^ 1(Y,\Sigma_{-\nu})$.
\end{remark}

\subsection{Auxiliary results}
The proof of our result will be a consequence of the following sequence of lemmas.

\begin{lemma}\label{cup-product}
The bilinear pairing
$$H^ 1(Y,\Sigma_{-\nu})\times H^ 1(Y,\Sigma_{\nu})\rightarrow H^2(Y,\C)=\C$$
given by the cup-product is non-degenerate. In particular, if  $\lbrace y_i\rbrace$ is a $\mu$-foliated system of coordinates (with respect to $\nu$) and if $\lbrace a_{ij}^{(\nu +1)}\rbrace $, seen as a cocycle in $Z^ 1(\mathcal V, NY^{\otimes{-\nu}})$, is not cohomologous to $0$,  then there exists $[b_{ij}]\in    H^ 1(Y,\Sigma_{\nu})$ whose cohomology class is trivial in $H^ 1(Y,NY^{\otimes \nu})$ and such that $[a_{ij}^{(\nu +1)}]\cup [b_{ij}]\not=0$.
\end{lemma}
\begin{proof}
One can see $[a_{ij}^{(\nu +1)}]$ as an ${NY}^{\otimes{-\nu}}$ valued harmonic form $\alpha\in  {\mathcal H }^1 (NY^{\otimes{-\nu}})$ whose $(0,1)$ part $\Omega$ is non trivial. One can then write $\Omega=\overline{\omega }$ where $\omega\in     {H }^{1,0} (Y,NY^{\otimes{\nu}})$. The form $\omega$ corresponds to a cocycle $[b_{ij}]\in    H^ 1(Y,\Sigma_{\nu})$ satisfying the conclusion of the lemma.
\end{proof}

\begin{lemma}\label{coordinatemodif}
 Let $\lbrace y_i\rbrace$ be a $\mu$-foliated system of coordinates which satisfies Equation (\ref {coefconstants}) and $\lbrace z_i\rbrace$ another admissible system of coordinates related to $\lbrace y_i\rbrace$ by
 $$y_i= z_i-H_i {z_i}^{\mu +1}$$
where $H_i\in  {\mathcal O}(\widehat{U_i})$. Then $\lbrace z_i\rbrace$ is still a $\mu$-foliated system of coordinates and the changes of coordinates are given by:
$$t_{ij} z_j-z_i=a_{ij}^{(\nu+1)}{z_i}^{\nu +1}+......+a_{ij}^{(\mu)}{z_i}^{\mu}+b_{ij}^{(\mu +1)}{z_i}^{\mu +1}$$
where the functions $b_{ij}^{(\mu +1)}\in   {\mathcal O}(\widehat{U_i}\cap \widehat{U_j})$ satisfy
$${ b_{ij}^{(\mu +1)} }_{|Y}= {(a_{ij}^{(\mu +1)}+t_{ij}^{-\mu} H_j -H_i)}_{|Y}.$$
\end{lemma}
\begin{proof}
It is a straightforward computation.
\end{proof}

Consider the coboundary in $Z^2(\mathcal V, N_Y^{-\mu})$ defined as
$${a_{ijk}^{(\mu +1)}}= {(a_{ij}^{(\mu +1)} +t_{ij}^{-\mu} a_{jk}^{(\mu +1)}+t_{ik}^{-\mu} a_{ki}^{(\mu +1)})}_{|Y}.$$
It is \emph{a priori} given by a collection of holomorphic functions but the following lemma shows that these functions are actually locally constant.
\begin{lemma}\label{2-cocycle}
Let $\lbrace y_i\rbrace$ be a $\mu$-foliated system of coordinates which satisfies Equation (\ref {coefconstants}).
\begin{enumerate}
 \item Whenever $\mu\geq 2\nu +1$, there exists a universal polynomial
$P_{\mu +1}\in \C[X_{ij}, X_{jk}, X_{ik}, Y_{ij}^{(l)}, Y_{jk}^{(l)}, Y_{ik}^{(l)}], {\nu+1}\leq l\leq \mu-\nu$ such that
$${a_{ijk}^{(\mu +1)}}=P _{\mu +1}(t_{ij}, t_{jk}, t_{ik}, a_{ij}^{(l)}, a_{jk}^{(l)}, a_{ik}^{(l)})$$ $$+(\nu-\mu-1)a_{ij}^{(\nu +1)} t_{ij}^{-(\mu-\nu )} a_{jk}^{(\mu-\nu +1)}-\nu a_{ij}^{(\mu-\nu +1)}  t_{ij}^{-\nu} a_{jk}^{(\nu +1)}.$$
\item If $\nu<\mu<2\nu $, then $${a_{ijk}^{(\mu +1)}}=0.$$
\item If $\mu=2\nu $, then
$${a_{ijk}^{(\mu +1)}}= -(\nu +1)a_{ij}^{(\nu +1)} t_{ij}^{-\nu} a_{jk}^{(\nu +1)}.$$
\end{enumerate}
In particular, ${a_{ijk}^{(\mu +1)}}$ is always  {locally constant} and is hence well defined as an element of $Z^2(\mathcal V, \Sigma_{-\mu})$.
\end{lemma}
\begin{proof}
Set $\alpha_{ij}= t_{ij}y_j-y_i$. Consider
$$A_{ijk}=\alpha_{ij}+t_{ij}\alpha_{jk}+t_{ik}\alpha_{ki}=\alpha_{ij}+t_{ij}\alpha_{jk}-\alpha_{ik}$$
and expand this expression with respect to $y_i$. As $\alpha_{ij},\,\alpha_{ik}$ are already expressed as a polynomial in the variable $y_i$, it is enough to  work with the middle term $\alpha_{jk}$, replacing $y_j$ by $y_i$ in accordance with Equation (\ref{coefconstants}). Using that $\lbrace y_i\rbrace$ is $\mu$-foliated, one easily observes that
$$A_{ijk}=\alpha_{ijk}^{(\nu +1)}y_i^{\nu +1}+....+\alpha_{ijk}^{(\mu )}y_i^{\mu}$$ $$+ (a_{ij}^{(\mu +1)} +t_{ij}^{-\mu} a_{jk}^{(\mu +1)}+t_{ik}^{-\mu} a_{ki}^{(\mu +1)}-\alpha_{ijk}^{(\mu +1)})y_i^{\nu +1}$$
where $\alpha_{ijk}^{(l)}$ is {locally constant} for $\nu +1\leq l\leq \mu$ and $\alpha_{ijk}^{(\mu +1)}\in\mathcal O (\widehat{U_i}\cap \widehat{U_j}\cap \widehat{U_k})$ with the additional property that ${\alpha_{ijk}^{(\mu +1)}}_{|Y}$ is locally constant.

On the other hand, $A_{ijk}$ vanishes identically. This forces the equality
$${\alpha_{ijk}^{(\mu +1)}}_{|Y}={a_{ijk}^{(\mu +1)}}$$
to hold. The expansion of $A_{ijk}$ with respect to $y_i$ allows us to express explicitly ${\alpha_{ijk}^{(\mu +1)}}_{|Y}$ and we obtain in that way the expected result.
\end{proof}

Denote by  $[{a_{ijk}^{(\mu +1)}}]$ the cohomology class of  ${a_{ijk}^{(\mu +1)}}$ in $H^2(Y, \Sigma_{-\mu})$.
\begin{lemma}\label{if0inH2}
Let $\lbrace y_i\rbrace$ be a $\mu$-foliated system of coordinates. Assume moreover that  $[{a_{ijk}^{(\mu +1)}}]=0$; then there exists a $(\mu +1)$-foliated system of coordinates $\lbrace z_i\rbrace$ such that
$$y_i=z_i+O({z_i}^{\mu +1}).$$
\end{lemma}
\begin{proof}
One can  write ${a_{ijk}^{(\mu +1)}}=\alpha_{ij}+t_{ij}^{-\mu} \alpha_{jk}+t_{ik}^{-\mu}\alpha_{ki}$ with $\alpha_{ij},\,\alpha_{jk}$ and $\alpha_{ik}$ locally constant. This means that $a_{ij}^{(\mu+1)}- \alpha_{ij}$ defines a cocycle in $Z^1(\mathcal V, {NY}^{\otimes -\mu})$. This cocycle is cohomologous to a locally constant one; in other words there exist $h_i\in \mathcal{O}(V_i)$ such that
$$e_{ij}:=a_{ij}^{(\mu+1)}+t_{ij}^{-\mu}h_j -h_i$$
is {locally constant}. Using Lemma \ref{coordinatemodif}, we are done setting
$$y_i=z_i-H_i {z_i}^{\mu+1}$$
where $H_i\in\mathcal O (\widehat{U_i})$ and coincides with $h_i$ when restricted to $V_i$.
\end{proof}

\begin{lemma}\label{ifleq2nu}
 Let $\lbrace y_i\rbrace$ be a $\mu$-foliated coordinate. Assume moreover that $\mu\leq 2\nu$. Then $[{a_{ijk}^{(\mu +1)}}]=0$.
\end{lemma}
\begin{proof}
By Lemma \ref{2-cocycle}, this is obvious if $\mu< 2\nu$. When $\mu=2\nu$, the same lemma shows that $[{a_{ijk}^{(\mu +1)}}]$ coincides with the cup-product $-(\nu +1)[a_{ij}^{(\nu +1)} ]^2=0$.
\end{proof}

\begin{lemma}\label{L:***}
Let $\lbrace y_i\rbrace$ be a $\mu$-foliated system of  coordinates written with the notation of Equation (\ref{coefconstants}) and assume that  $\mu\geq 2\nu +1$, then there exists
an admissible system of coordinates $\lbrace z_i\rbrace$ of the form $y_i=z_i+ O({z_i}^{\mu-\nu+1})$ such that $\lbrace z_i\rbrace$ is a $\boldsymbol {(\mu +1)}$-{foliated system of coordinates}.
\end{lemma}
\begin{proof}
According to Lemma \ref{if0inH2}, we are done whenever $[{a_{ijk}^{(\mu +1)}}]=0$. This occurs in particular if $N_Y^{-\mu}$ is non trivial, in virtue of Lemma \ref{vanishingH2}. One can then assume that $N_Y^{-\mu}$ is trivial and that $[{a_{ijk}^{(\mu +1)}}]\neq0$.  Note that a simple change of $\mu$-coordinates of the form $y_i=z_i+O(z_i^{\mu +1})$ does not affect  $[{a_{ijk}^{(\mu +1)}}]\not=0$, this justifies that we have to act retroactively, as stated in the lemma, in order to modify suitably this cohomology class.

To this end, let us consider $\lbrace \alpha_{ij}^{}\rbrace \in Z^1(\mathcal V, \Sigma_{-(\mu -\nu)})$ which is a coboundary when seen as a cocycle in $Z^1(\mathcal V, {NY}^{\otimes{-(\mu -\nu)}})$ and such that
$$[{a_{ijk}^{(\mu +1)}}]=-(2\nu-\mu-1)[a_{ij}^{(\nu +1)}]\cup [\alpha_{ij}^{}].$$
Remark that $\Sigma_{-(\mu -\nu)}=\Sigma_{\nu}$, hence such an $\alpha_{ij}$ exists, according to Lemma \ref{cup-product}.

Let $h_i\in {\mathcal O}(V_i)$ such that $t_{ij}^{-(\mu-\nu)}h_j-h_i=\alpha_{ij}$. Consider $H_i\in{\mathcal O}(\widehat{U_i})$ such that ${H_i}_{|V_i}=h_i$. Define $\lbrace z_i\rbrace$ such that $y_i=z_i -H_i {z_i}^{\mu -\nu +1}$. By Lemma \ref{coordinatemodif}, $\lbrace z_i\rbrace$ is a $(\mu -\nu +1)$-foliated system of coordinates and more precisely,
$$t_{ij} z_j-z_i=a_{ij}^{(\nu+1)}{z_i}^{\nu +1}+......+a_{ij}^{(\mu-\nu)}{z_i}^{\mu-\nu}+b_{ij}^{(\mu -\nu+1)}{z_i}^{\mu-\nu +1}+b_{ij}^{(\mu -\nu+2)}{z_i}^{\mu-\nu +2}$$
where ${b_{ij}^{(\mu -\nu+1)}}$ is constant and equal to $a_{ij}^{(\mu -\nu+1)}+\alpha_{ij}$. If $\mu -\nu +2<\mu +1$, then one has automatically $ [b_{ijk}^{(\mu -\nu+2)} ]=[a_{ijk}^{(\mu -\nu+2)}]=0$. Actually, equality holds thanks to Lemma \ref{2-cocycle}, noticing that the coefficients $a_{ij}^{(l)},\ \nu+1\leq l \leq\mu -\nu$ associated to $\lbrace y_i\rbrace$ remain the same for $\lbrace z_i\rbrace$. Invoking again Lemma \ref{if0inH2}, one can assume that ${b_{ij}^{(\mu -\nu+2)}}_{|Y}$ is constant. One can repeat the same operation for ${b_{ijk}^{(\mu -\nu+3)} }$ provided that $\mu -\nu +3<\mu +1$. Proceeding inductively, we can eventually assume that $\lbrace z_i\rbrace$ is a $\mu$-foliated system of coordinates which satisfies the relation
$$t_{ij} z_j-z_i=a_{ij}^{(\nu+1)}{z_i}^{\nu +1}+......+a_{ij}^{(\mu-\nu)}{z_i}^{\mu-\nu}+b_{ij}^{(\mu -\nu+1)}{z_i}^{\mu-\nu +1}$$ $$+....+b_{ij}^{(\mu)}{z_i}^{\mu} +b_{ij}^{(\mu +1)}{y_i}^{\mu +1}$$
where $b_{ij}^{(\mu +1)}\in \mathcal O (\widehat{U_i}\cap \widehat{U_j})$.
Applying again Lemma \ref{2-cocycle}, and recalling that ${b_{ij}^{(\mu -\nu+1)}}=a_{ij}^{(\mu -\nu+1)}+\alpha_{ij}$, one obtains that
$$b_{ijk}^{(\mu +1)}=a_{ijk}^{(\mu +1)}+(\nu-\mu-1)a_{ij}^{(\nu +1)} t_{ij}^{-(\mu-\nu )} \alpha_{jk}-\nu \alpha_{ij}  t_{ij}^{-\nu} a_{jk}^{(\nu +1)}.$$
One can recognize a cup-product in the right hand side, namely
$$[b_{ijk}^{(\mu +1)}]=[a_{ijk}^{(\mu +1)}]+(2\nu-\mu-1)[a_{ij}^{(\nu +1)} ]\cup [\alpha_{ij}].$$
One concludes observing that $\alpha_{ij}$ has been chosen in such a way that $ [b_{ijk}^{(\mu +1)}]=0$. One thus obtains the desired  $(\mu +1)$-foliated system of coordinates.
\end{proof}

\subsection{Proof of Theorem \ref{THM:formalfoliation}}\label{sec:proofFormalFred}
Note firstly that when $\utype(Y)=\infty$ the existence of a formal foliation having $Y$ as a compact leaf
follows from the definition of Ueda type. Form now assume that $\nu=\utype (Y)$ is finite.
 In particular there exists a $(\nu +1)$-foliated system of coordinates $\lbrace y_i\rbrace$. Note also that $a_{ij}^{(\nu +1)}$ is nothing but the Ueda class and then satisfies the hypothesis of Lemma \ref{cup-product}. Combining Lemmas \ref{if0inH2} and \ref{ifleq2nu}, one can assume that $\lbrace y_i\rbrace$ is indeed a $(2\nu +1)$-foliated system of coordinates. The existence of a formal foliation follows  from Lemma \ref{L:***}
\qed

\begin{remark}
The formal foliation constructed above comes equipped with a holonomy representation  $\hat\rho:\pi_1(Y)\to\Difffor$ whose linear part is unitary and $\utype (Y) +1$ coincides with the first jet for which
the representation $\hat\rho$ fails to be linearizable. %(see lemma \ref{L:Uedaholonomy}).
\end{remark}

\subsection{Higher dimensional version of Theorem \ref{THM:formalfoliation}}

In the above proof, we use essentially the vanishing of the second cohomology group for non trivial unitary rank one local systems, and the surjectivity of the cup product map (from $H^1\wedge H^1$ to $H^2$). These properties are rather specific to the curve case. However, it is possible to prove the existence of formal foliation in some higher dimensional cases. In particular, it is possible to prove the following statement.

\begin{thm}\label{T:higher dimensional formal foliation}
Let $Y\subset X$ be a smooth hypersurface (compact K\"ahler) embedded in a germ of neighborhood $X$. Assume moreover that $NY=\mathcal{O}_Y$ and that $h^1(Y,\mathcal{O}_Y)=1$. If $Y(\infty)$ denotes the formal completion of $X$ along $Y$, then $Y$ is a leaf of a (formal) foliation $\hat{\F}$ on $Y(\infty)$ defined by a closed formal one form $\hat{\omega}$.
If $\nu=\utype(Y)$ is finite, then $\hat{\omega}$ can be described locally as follows:
$$\hat{\omega}=_{loc}\frac{\mathrm{d}y}{y^{\nu+1}}+\lambda\frac{\mathrm{d}y}{y},\ \ \ \lambda\in\C.$$
\end{thm}
\begin{proof}
As above, we can assume that $\nu:=\utype(Y)$ is finite and we consider first the case $\nu=1$. Let us choose $\lbrace y_i\rbrace$ a system of coordinates such that
$${y_i}= {y_j}-a_{ij}{y_j}^{2} +o({y_j}^{2})$$
or equivalently
$$\frac{1}{y_i}-\frac{1}{y_j}=a_{ij} +o(1)$$
where $[{a_{ij}}_{|Y}]$ defines a cocycle whose class (in $H^1( Y, \mathcal{O}_Y)$) is not trivial and which can be assumed to be locally constant (every cocycle is cohomologous to a constant one). One can then write
$$\frac{1}{y_i}-\frac{1}{y_j}=a_{ij} +b_{ij} {y_j}$$
with $n>0$ and $a_{ij}$ locally constant in $Y(\infty)$.

\noindent \textbf{Claim:} $\lbrace {b_{ij}}_{|Y}\rbrace$ is a cocycle with coefficients in $\mathcal{O}_Y$.\\
Indeed, differentiate the previous equality and set $\omega_{ij}=-\frac{dy_i}{{y_i}^2} +\frac{dy_j}{{y_j}^2} $. Let $k$ be such that $V_k\cap V_i\cap V_j\not=\emptyset$. Because $y_k=y_i + o(y_i )$, one obtains that ${\mbox{Res}}_Y \frac{\omega_{ij}}{{y_k}}$ is independent of $k$, hence defines a cocycle which is nothing but $\lbrace {b_{ij}}_{|Y}\rbrace$. \qed

Because $H^1(Y,\mathcal{O}_Y)$ is one dimensional, there exists $\lambda\in \C$ such that $[b_{ij}]=\lambda [a_{ij}]$. Let $a_{ij}^{(1)}\in {\mathcal O}(\widehat{U_i}\cap \widehat{U_j})$ be such that
$$\frac{1}{y_i}-\frac{1}{y_j}+\lambda\log\frac{y_i}{y_j}=a_{ij} +a_{ij}^{(1)} {y_j}.$$
The cocycle $\lbrace {a_{ij}^{(1)}}_{|Y}\rbrace$ is then cohomologous to $0$.

The proof can then be simply deduced from the following sequence of observations (whose proofs are omitted for the sake of brevity). As in the proof of Theorem \ref{THM:formalfoliation}, we use changes of coordinates in order to force obstructions to vanish inductively.

\begin{lemma}
Assume that for some $n>0$,
\begin{equation}\label{ordern}
\frac{1}{y_i}-\frac{1}{y_j}+\lambda\log\frac{y_i}{y_j}=a_{ij} +a_{ij}^{(n)} {{y_j}}^n
\end{equation}
then $\lbrace {a_{ij}^{(n)}}_{|Y}\rbrace$  is a cocycle.
\end{lemma}

\begin{lemma} Assume that the coordinates $\lbrace y_i\rbrace$ satisfy the relation (\ref{ordern}) with $n>1$. The cocycle $\lbrace {a_{ij}^{(n)}}_{|Y}\rbrace$ is then cohomologous to $\lambda^{(n)}[a_{ij}]$ (for some $\lambda^{(n)}\in\C$). In the new system of coordinates given by
$$z_i=y_i-\frac{\lambda^{(n)}}{n-1}{y_i}^{n+1},$$
the relation (\ref{ordern}) becomes
$$\frac{1}{z_i}-\frac{1}{z_j}+\lambda\log\frac{z_i}{z_j}=a_{ij} +\tilde{a}_{ij}^{(n)} {{z_j}}^n$$
where $\lbrace \tilde{a}_{ij\vert Y}^{(n)}\rbrace$ is cohomologous to zero.
\end{lemma}

\begin{lemma}
Assume that the coordinates $\lbrace y_i\rbrace$ satisfy the relation (\ref{ordern})  with $ \lbrace {a_{ij}^{(n)}}_{|Y}\rbrace$ cohomologous to zero. Write ${a_{ij}^{(n)}}_{|Y}=a_i-a_j$ and consider the transformation $$\frac{1}{z_i}= \frac{1}{y_i}-A_i {y_i}^n$$ with $A_i\in{\mathcal O}(\widehat{U_i})$ such that ${A_i}_{|Y}=a_i$. The relation (\ref{ordern}) can then be read
$$\frac{1}{z_i}-\frac{1}{z_j}+\lambda\log\frac{z_i}{z_j}=a_{ij} +a_{ij}^{(n+1)} {{z_j}}^{n+1}.$$
\end{lemma}

Noticing that the order (with respect to $y_i$) of the transformations which are involved in these successive reductions increases with $n$, one obtains (in the limit) the sought formal foliation. If the Ueda type satisfies $\nu>1$, the same proof works when considering the quantity
$$\frac{1}{y_i^\nu}-\frac{1}{y_j^\nu}=a_{ij}+o(1).$$
This concludes the proof of Theorem \ref{T:higher dimensional formal foliation}.
%(see the proof of Theorem \ref{THM:Existence}).
\end{proof}

\section{Existence of global foliations} \label{S:Existence2}
In this section we turn our attention to the existence problem of foliations with prescribed compact leaves on compact K\"ahler manifolds.

\subsection{Proof of Theorem \ref{THM:Existence}}\label{S:UedaLorayNeeman}
Let $Y \subset X$ be a smooth  divisor on a compact K\"{a}hler manifold  with normal bundle $NY$. Let $\mathcal L$ be a line bundle on $X$, topologically torsion, such that $NY^{\otimes k} = \mathcal L_{|Y}$. We will denote by $\nabla$ the unique flat unitary connection (identified with its $(1,0)$ part) on $\mathcal L^*$.

Since $\utype(Y)\ge k$,
there exists an open covering $\{ U_i\}$ of $X$ and  holomorphic functions $y_i \in \mathcal O_X(U_i)$
such that
 $Y\cap U_i = \{ y_i =0 \}$; and
 if $U_i \cap U_j \cap Y \neq \emptyset$, then
\begin{equation}\label{E:ueda}
y_i^k = ( \lambda_{ij} + y_j^k r_{ij}   ) y_j^k \,
\end{equation}
for some constants $\lambda_{ij} \in S^1$ and some holomorphic functions $r_{ij}  \in \mathcal O_U(U_i\cap U_j)$.
Of course the cocycle $\{\lambda_{ij}\}$ represents the line-bundle $\mathcal L$ in $H^1(X,S^1)$.

Whenever $U_i \cap U_j \neq \emptyset$, set
\begin{equation}\label{E:cocyle}
a_{ij}  = \frac{1}{y_i^k} - \frac{1}{\lambda_{ij}} \frac{1}{y_j^k}.
\end{equation}
Equation (\ref{E:ueda})
implies that $a_{ij} \in \mathcal O_X(U_i\cap U_j)$. The collection $\{ a_{ij} \}$ determines
an element  $\alpha$ of $H^1(X, \mathcal L^*)$.

One can then write
$$a_{ij} =h_i-h_j$$
where the $h_i$'s are {\it smooth} local sections of $\mathcal L^*$. Note that the collection of $\overline{\partial} h_i$ defines a globally defines $(0,1)$ form $\eta$ valued in $\mathcal L^*$. Consequently, $\xi=\nabla \eta$ is closed with respect to $\nabla+\overline{\partial}$. By harmonic theory on compact K\"{a}hler manifold\footnote{This can be done exactly along the same lines than the classical $\partial\overline{\partial}$ lemma, replacing the $\partial$ operator by the unitary flat connection $\nabla$.}, $\xi$ is not only $\nabla$-exact but is actually $\nabla \overline{\partial}$-exact. One can then assume that $\xi=0$. This obviously implies that one can locally express the differential of the cocycle $a_{ij}$ as
  \[
d a_{ij} =  \omega_i -  \frac{1}{\lambda_{ij}} \omega_j
\]
where $\omega_i$ is a closed {\it holomorphic} form. This proves that
\[
\Omega=d (\frac{1}{y_i^k})- \omega_i = \frac{1}{\lambda_{ij}} \left( d (\frac{1}{y_j^k})- \omega_j \right)
\]
defines a $\nabla$-closed rational section of $\Omega^1_X \otimes \mathcal L^*$ with
polar divisor equal to $(k+1) Y$.
 \qed

%\subsection{Proof of Theorem \ref{THM:Existence}}\label{S:UedaLoray}

\begin{cor} Let  $Y \subset X$ be a smooth  divisor on a compact K\"{a}hler manifold  with normal bundle torsion of order $k$ and $\utype(Y)\ge k$. Then there exists a closed rational form on $X$ with poles of order $k+1$ on $Y$ which defines a foliation having $Y$ as a compact leaf.
\end{cor}
\begin{proof}
Apply Theorem \ref{THM:Existence} with  $\mathcal{L}=\mathcal{O}_X$
\end{proof}
\begin{remark}\label{r:localglobal}
In this case (see proof of Theorem \ref{THM:Existence}), one can notice that the Ueda class, which lies in $H^1(Y, {\mathcal O}_Y)$, is a restriction of some global class in $H^1(X, {\mathcal O}_X)$.
\end{remark}

Recall that the case $\utype (Y)>k=\ord(NY)$ gives more specific information and is covered by Neeman's result  (Theorem \ref{thm:Neeman} stated in the introduction). For the sake of completeness, we give a simplified proof of this statement.
\subsection{Proof of Theorem \ref{thm:Neeman}}
Assume that $\utype(Y)> k$ and let us consider (as in the  proof of Theorem \ref{THM:Existence}) an open covering $\{ U_i\}$ of $X$ and  holomorphic functions $y_i \in \mathcal O_X(U_i)$
such that
 $Y\cap U_i = \{ y_i =0 \}$; and
 if $U_i \cap U_j \cap Y \neq \emptyset$ then
\begin{equation}\label{E:uedabis}
y_i^k = ( 1 + y_j^k r_{ij}   ) y_j^k \,
\end{equation}
for some holomorphic function $r_{ij}  \in \mathcal O_U(U_i\cap U_j)$.

Whenever $U_i \cap U_j \neq \emptyset$, set
\begin{equation}\label{E:cocylebis}
a_{ij}  = \frac{1}{y_i^k} - \frac{1}{y_j^k}.
\end{equation}
Equation (\ref{E:uedabis})
implies that $a_{ij} \in \mathcal O_X(U_i\cap U_j)$. The collection $\{ a_{ij} \}$ determines
an element  $\alpha$ of $H^1(X, \mathcal O_X)$.

The cohomology class  $-k\alpha_{|Y} \in H^1(Y,\mathcal O_Y)$ coincides with Ueda's original
definition\footnote{Recall that $\alpha$ is only well-defined up to  multiplication by a complex constant
which comes from the choice of an isomorphism between $\mathcal O_Y$ and $\mathcal I^k/ \mathcal I^{k+1}$.} of the $k$-th Ueda class of $Y$ in $X$. The assumption $\utype(Y)>k$ implies that $-k \alpha_{|Y} = 0$ and to conclude the proof we argue differently according to the injectivity (or not) of the restriction map $H^1(X,\mathcal O_X) \to H^1(Y,\mathcal O_Y)$.

If $H^1(X,\mathcal O_X) \to H^1(Y,\mathcal O_Y)$ is not injective then  Proposition \ref{P:Albanese} implies that  $kY$ is a fiber of a fibration and we have $\utype(Y)=\infty$.

If $H^1(X,\mathcal O_X) \to H^1(Y,\mathcal O_Y)$ is injective, $\alpha$ is then zero in $H^1(X,\mathcal O_X)$ and we can write $a_{ij} = {a_i}_{|U_i \cap U_j} - {a_j}_{|U_i \cap U_j}$
for suitable $a_i \in \mathcal O_X(U_i)$. Hence, we can construct a morphism  $f: X \to \mathbb P^1$ such that $f^{-1}(\infty) = k Y$ by setting
$f_{|U_i} = y_i^{-k} - a_i$. It follows that  $\utype(Y) = \infty$, and that $kY$ is a fiber of a fibration.
\qed

\begin{remark}
The smoothness of the divisor is not really important in the proof of Theorem \ref{thm:Neeman}.
In particular, the argument here presented  can be easily adapted to give an alternative proof of  \cite[Theorem 2.3]{croco5}.
\end{remark}

\subsection{Hypersurfaces with non-torsion line-bundles}
Below is a version of Theorem \ref{thm:Neeman} for non-torsion normal bundles, also due to Neeman \cite[Theorem 5.6, p. 113]{Neeman}.

\begin{thm}\label{thm:neemanextension}
Let $Y$ be a smooth and irreducible hypersurface of a compact K\"{a}hler  manifold $X$. Assume that $NY=\mathcal O_X(Y)_{|Y}$ has topologically torsion normal bundle
but has analytically infinite order.
If $\utype(Y)=\infty$ and for some positive integer $k>0$, the $k$-th power  $\rho^k$ of the unitary representation $\rho:\pi_1(Y) \to S^1$ associated to $NY$
extends to a representation $\tilde \rho: \pi_1(X) \to S^1$, then
there exists an integral effective  divisor $D$, disjoint from $Y$ and cohomologous to $kY$ for some positive integer .
In particular, at a sufficiently small Euclidean neighborhood of $Y$, the Ueda line bundle  $\mathcal U$ is  trivial.
\end{thm}

\begin{example}\label{exampleOgusFred}
Consider a genus $g>1$ curve $C$ and a representation $\underline{\rho}:\pi_1(C)\to \mathrm{Aut}(\mathbb P^1)$
into the unitary dihedral subgroup with infinite image: the elements are all elliptic, permuting $\{0,\infty\}\subset\mathbb P^1$.
On a $2$-fold (\'etale) cover $\tilde C\stackrel{2:1}{\to}C$, the representation lifts as $\rho:\pi_1(\tilde C) \to S^1$ with infinite image
(fixing $0$ and $\infty$).
Considering the suspension of $\underline{\rho}$, we get a Riccati foliation $\mathcal F$ on a ruled surface $\pi:X\to C$
with an invariant $2$-section $X\supset Y\stackrel{2:1}{\to}C$. After fiber product
$$ \xymatrix{
    \tilde X \ar[r]^-{2:1} \ar[d]_{\tilde\pi}  &  X\supset Y \ar@<-10pt>[d]^\pi \\
    \tilde C \ar[r]^-{2:1} & C\hskip0.6cm
  }$$
we get the total space $\tilde\pi:\tilde X\to \tilde C\simeq Y$ of a non-torsion line bundle on which $Y$ lifts as
the zero and infinity sections; the germ of neighborhood $(X,Y)$ is isomorphic to the germ $(\tilde X,\tilde Y_0)$
at the zero section. The Ueda line bundle is therefore trivial at a neighborhood of $Y$ where we have local flat analytic structure.
However, no power of the local holonomy representation of $\pi_1(Y)$ extends to a representation of $\pi_1(X)$.
The complement $X- Y$ contains no complete curve since such a curve would lift as a finite section,
contradicting that $NY$ is non torsion. Moreover,  $X- Y$ is not Stein. Indeed by virtue of the maximum principle, every $f\in\mathcal {O}(X-Y)$ is constant on the closure of the leaves in $X-Y$ (which are compact Levi-flat hypersurfaces). This implies that $f$ is actually {\it constant}.
In particular, this answers negatively   \cite[Chapter VI, Problem 3.4]{Hartshorne}.

As already mentioned the previous example is a variation of an example by Ogus \cite[Section 4]{Ogus}, where the ambient $X$ is singular but rational.
The construction is similar. It starts with $\tilde\pi:\tilde X\to \tilde C$, the total space of a non torsion line bundle over an elliptic curve $\tilde C$,
equipped with a holomorphic connection, and then form the quotient $X$ by a lift of the elliptic involution. There are $8$ fixed points
away from the zero and infinity sections, giving rise to singular points for $X$. Take $Y$ to be the image of any of the two sections of $\tilde\pi$.

Another interesting example,  which is somehow opposite to Hartshorne question,  is provided by the surface $S$ obtained by blowing up $9$ points on a smooth cubic curve $C_0$ on ${\mathbb P}^2$. For a generic choice of these $9$ points, related to the Diophantine condition stated in Theorem \ref{T:Uedainfty}, the strict transform $C$ is left invariant by a regular foliation (only defined in its neighborhood) such that the leaves closure are compact Levi-flat hypersurfaces.  This prevents $S-C$ from being Stein. However, it is still unknown whether there exists  some exceptional configuration of points for which $S-C$ is Stein (see the discussion in \cite{Brunella2}).
\end{example}

\subsection{Curves with torsion normal bundle}\label{S:QE} Theorem \ref{THM:Existence} says nothing about the existence of
smooth foliations along $Y$ when $\utype(Y)< \ord(NY)$. In particular, it leaves open the following natural question:

\begin{question}\label{Q:existence}
Is every curve $C \subset S$  with torsion normal bundle on a projective surface $S$ a compact leaf of a foliation?
\end{question}

Given any curve $C$ of genus $g(C)>1$, it is relatively easy to construct smooth foliations on projective surfaces
having $C$ as a compact leaf satisfying $\utype(C)=1$  and $\ord(NC)$ arbitrary. If we take a general representation
$\rho:\pi_1(C) \to \Aff(\mathbb C)$ with linear part of order $k$ then the Riccati foliation obtained through the suspension
of $\rho$ is an example with $\utype(C)=1$ and $\ord(NC)=k$. Taking some Kawamata coverings of these examples (see \cite[Proposition 4.1.12]{Lazarsfeld}), one can obtain
examples with $\utype(C) = n$ and $\ord(NC)=kn$ for arbitrary $k$ and $n$.

Perhaps one of the simplest tests for Question \ref{Q:existence} is the existence of foliations along a smooth curve $C$  of genus
$3$ with $\utype(C) = 1$ and $\ord(NC)=2$ constructed in the following way, see \cite[page 119]{Neeman}. Start with a smooth
quartic $C_0 \subset \mathbb P^2$ and consider two bitangents to it. The two bitangents determine $4$ points on $C_0$. Choose twelve other
points as the intersection of $C_0$ with a generic cubic. The blow-up of $\mathbb P^2$ on these $16$ points gives a rational surface
containing a curve $C$, the strict transform of $C_0$, having torsion normal bundle of order two and Ueda type equal to one. We do not
know if there exists a foliation on $S$ regular along $C$.

\section{Groups of (formal) germs of biholomorphisms}\label{S:subgroups}
In this section we review the  bits of the theory of subgroups of $\Diff$ and $\Difffor$ relevant to our study.
The only  new result in this section is Proposition \ref{P:utypek}.

\subsection{Formal classification of solvable subgroups}\label{S:formalclassification}
  Let $\Diff$ denotes the group of germs of analytic diffeomorphisms of one variable fixing $0\in\C$ and $\Difffor$ its formal completion. In order to state the classification of formal subgroups of $\Difffor$, it will be convenient to use the vector fields (or rather theirs flows/exponentials)
\[
v_{k,\lambda} = \frac{z^{k+1}}{1+ \lambda {z^k}} \frac{\partial}{\partial z}
\]
for $k \in \mathbb N^*$ and $\lambda \in \mathbb C$. Note that  $h= \exp\left( v_{k,\lambda}\right)\in \Diff$ is tangent to identity at order $k$, \emph{i.e.} $h(z)= z+z^{k+1}+\cdots$.  We recall firstly the classification of elements of $\Difffor$ up to formal conjugacy (see for instance \cite[\S 1.3]{Frankpseudo} ).

\begin{thm}
  Let $f(z)=az+...\in \Difffor,a\in{\C}^*$, be distinct from the identity. Then $f$ is conjugated to one and only one of the following models:
  \begin{enumerate}
\item $f_0 (z)=az$ where $a\in{\C}^*$.
\item $f_0(z)=a \cdot \exp\left( v_{k,\lambda}\right); a^k =1$.
\end{enumerate}
\end{thm}
In the first case $f$ is said to be (formally) \textit{linearizable}. In the second item, remark that $z\rightarrow az$ commutes with $\exp\left( v_{k,\lambda}\right)$ and that $f_0$ (hence $f$) has infinite order.

The formal classification of solvable subgroups of $\Difffor$ is given in \cite{CeMo}; we also refer to \cite[\S 1.4]{Frankpseudo} and to the recent monograph \cite{CaCeDeBook}.

 \begin{thm}\label{T:formalG}
Let $G \subset \Difffor$ be a subgroup. If $G$ is virtually solvable, then
there exists $ {\varphi} \in \Difffor$ such that $ \varphi_* G$ is a subgroup of one of the following models:
\begin{enumerate}
\item $\mathbb L = \{ f(z) = a z; a \in \mathbb C^* \}$;
\item $\mathbb E_{k,\lambda} = \{ f(z) = a \cdot \exp\left(t v_{k,\lambda}\right); a^k =1 , t \in \mathbb C\}$, for some $k \in \mathbb N^*$ and $\lambda \in \mathbb C$;
\item $\mathbb A_k = \{ f(z) = az / (1 - b z^k)^{1/k}; a \in \mathbb C^*, b\in \mathbb C\}$, for some $k\in \mathbb N^*$.
\end{enumerate}
\end{thm}

When $G$ is abelian, we are in cases (1) or (2). Indeed, abelian subgroups of $\mathbb A_k$ actually lie (up to conjugacy) in
$\mathbb L$ or $\mathbb E_{k,0}$.

\begin{definition}\label{def:filtration-diff}
Let $\Difffor_1=\{h\in \Difffor \mid h(z)=z+o(z)\}$ be the subgroup of $\Difffor$ whose elements are tangent to identity. More generally, if $k\ge1$, let us define
$$\Difffor_k=\{h\in \Difffor \mid h(z)=z+o(z^k)\}.$$
\end{definition}

Notice that for  $h\in\Difffor_1$, the following assertions are equivalent: $h=Id$;  $h$ has finite order; and  $h$ is linearizable.
More generally, we have the useful linearizability criterion, see \cite[Corollaire 1.4.2]{Frankpseudo}.

\begin{thm}\label{thm:triviallinearpart}
Let $G$ be a subgroup of $\Difffor$. Set $G_1=G\cap \Difffor_1$. Then, up to formal conjugacy,  $G$ is a subgroup of $\mathbb L$ if and only if $G_1=\{Id\}$.
\end{thm}
 \begin{cor}\label{cor:finiteorder}Let $G$ be a subgroup of $\Difffor$.  Then every element of $G$ has finite order if, and only if, $G$ is conjugated to a subgroup of
 $${\mathbb L}_\mathbb Q :=\{ f(z) = a z; a \in e^{2i\pi\mathbb Q} \}.$$
 \end{cor}

\subsection{Non solvable subgroups}
 For further use, we recall the characterization of non solvable subgroups of $\Difffor$ \cite[Theorem 1.4.1] {Frankpseudo}.
 \begin{thm}
 Let $G$ be a subgroup of $\Difffor$. Then the following properties are equivalent:
 \begin{enumerate}
 \item $G$ is non solvable.
 \item $G$ is non virtually solvable.
 \item $G$ is not metabelian.
 \item For every $k > 0 $, we have $G\cap\Difffor_k\not=\{Id\}$, \emph{i.e.}, there exist non trivial elements of $G$
tangent to the identity at arbitrarily large order.
 \end{enumerate}
 \end{thm}

 \subsection{Centralizers}
 Let  $h\in \Difffor$ and denote by $\mathcal C (h)$ the centralizer of $h$ in $\Difffor$.
 \begin{thm}[\cite{Frankpseudo} Proposition 1.3.2]\label{thm:centralizers}
 The group $\mathcal{C}(h)$ is given by
 \begin{enumerate}
 \item If $h (z)=az$ where $a\in{\C}^*$ is not a root of unity, then ${\mathcal C} (h)=\mathbb L$.
%\item $f_0(z)=az$ if $a^q=1$ and $f^{\circ q}=Id$ for some positive integer $q$.
\item If $h(z)=a \cdot \exp\left( v_{k,\lambda}\right)$ with $a^k =1$, then ${\mathcal C}(h)= \{e^{2i\pi p/k}\cdot \exp\left(t v_{k,\lambda}\right),\ p\in\mathbb Z, t\in\C \}$
\end{enumerate}
\end{thm}

\begin{thm}\label{thm:center}
Let $G$ a subgroup of $\Difffor$. Assume in addition that $G$ is non virtually abelian. Then the center $Z(G)$ of $G$ is finite, hence conjugated in $\Difffor$ to a finite subgroup of $\mathbb L$.
\end{thm}
\begin{proof}
When $G$ is non solvable, this is covered by \cite[Proposition 1.5.1]{Frankpseudo}. Let us assume that $G$ is solvable and non abelian. Up to conjugation, $G$ is a subgroup of $\mathbb A_k$ such that $G_1$ is non trivial. Hence it contains an element conjugated to $g_1= \exp\left( v_{k,\lambda}\right) $. Now let us choose $h\in Z(G)$: in particular $h$ lies in $\mathcal{C} (g_1)$ and $G$ is contained in $\mathcal{C}(h)$. According to Theorem \ref{thm:centralizers}, there exists $t_h\in\C$ such that $h=a \cdot \exp\left( t_hv_{k,\lambda}\right)$ with $a^k =1$. Assuming that $t_h\not=0$ and using Theorem \ref{thm:center} once more, one can infer that $G\subset {\mathcal C} (h)$ is abelian. Since we assumed this was not the case from the beginning, we get $t_h=0$ and $Z(G)$ is necessarily conjugated to a subgroup of $\{z\rightarrow e^{2i\pi p/k}z,\ p\in\mathbb Z \}$.
\end{proof}

\begin{remark}\label{r:tangenttoid}
  If $G$ is a subgroup of $\Difffor_1$, the following properties are equivalent:
  \begin{itemize}
  \item[$(i)$] $G$ is abelian.
  \item[$(ii)$] $G$ is virtually abelian.
  \item[$(iii)$] $G$ is solvable.
  \item[$(iv)$] The non trivial elements of $G$ are tangent to identity to the same order.
  \end{itemize}
\end{remark}

\subsection{Interpretation}\label{S:interpretation}
The groups appearing in Theorem \ref{T:formalG} are  groups of symmetries of meromorphic $1$-forms   or of vector spaces of meromorphic $1$-forms.

The group $\mathbb L$ is the group of symmetries of the logarithmic $1$-form $\omega:=\frac{dz}{z}$, while for fixed $k \in \mathbb N^*$ and $\lambda \in \mathbb C$, the group $\mathbb E_{k,\lambda}$ is precisely the group of symmetries
of the  $1$-form
\[
\omega_{k,\lambda} =  \frac{dz}{z^{k+1}}   + \lambda \frac{dz}{z}  \, .
\]
Moreover, for a fixed $k \in \mathbb N^*$, the group $\mathbb A_k$ is the subgroup   of $\Difffor$ that preserves the vector space
of $1$-forms $\mathbb C \frac{dz}{z^{k+1}}$. Notice that $\mathbb E_{k,0}$ is a
subgroup of $\mathbb A_k$.

\begin{remark}\label{rk:abelianhol}
It follows from the above description and Theorem \ref{T:formalG} that an abelian subgroup
$G\subset\Diff$ preserves a formal meromorphic $1$-form $\omega$. Consequently,
if $G$ is the holonomy group of the compact leaf
$Y$ of a foliation on a complex manifold $U$, then by standard arguments using flow-box, the $1$-form $\omega$
extends to a formal meromorphic $1$-form on the completion $Y(\infty)$ of $U$ along $Y$.
\end{remark}

\subsection{Ueda type and formal closed meromorphic $1$-forms}
If $Y$ is a hypersurface of a complex manifold $U$, the existence of formal closed meromorphic $1$-forms on $Y(\infty)$ gives
information about the Ueda type of $Y$.

\begin{prop}\label{P:utypek}
Let $Y$ be a smooth irreducible compact  hypersurface on a complex manifold $U$. Let $\hat{\omega} \in H^0(Y(\infty), \Omega^1_{Y(\infty)} ( (k+1) Y))$
be a formal closed meromorphic $1$-form on $Y(\infty)$ with poles of order $k+1$ on $Y$.
If $k\ge 1$ then $\utype(Y)\ge k$ and the order of the
normal bundle of $Y$ divides $k$.
\end{prop}
\begin{proof}
The $1$-form $\hat \omega$  can be locally written as
$ \frac{dz}{z^{k+1}}   + \lambda \frac{dz}{z}$ where $z$ is a formal submersion cutting out $Y$. More precisely, there exists an
open covering $U_i$ of $Y(\infty)$ and $z_i \in H^0(U_i, \mathcal O_{Y(\infty)})$  such that $\hat \omega_{|U_i} = \frac{dz_i}{z_i^{k+1}}   + \lambda \frac{dz_i}{z_i}$
and $Y \cap U_i = \{ z_i = 0\}$. It follows from the interpretation of $\mathbb E_{k,\lambda}$ given in Section \ref{S:interpretation} that, once the covering is fixed, the only freedom we have in the choice of
the functions $z_i$ is left composition by elements of $\mathbb E_{k,\lambda}$. Therefore, over non-empty intersections $U_i\cap U_j$,  the formal functions $z_i$ satisfy
\[
z_i  =  ( \lambda_{ij} + a_{ij} z_j^k + b_{ij} z_j^{2k} + \ldots )\cdot z_j
\]
with $\lambda_{ij}^k = 1$ and $a_{ij}, b_{ij} \in \mathbb C$. The proposition follows from Lemma \ref{L:Uedaholonomy}.
\end{proof}

\subsection{Rigidity} The analytic classification of subgroups of $\Diff$ does not coincide with
the formal classification. Nevertheless, for a large class of groups we have the following  rigidity statement
proved in \cite[Proposition 1]{CeMo}.

\begin{thm}
Let $G$ be a subgroup of $\Diff$ and let $G_1\subset G$ be the subgroup of elements with trivial linear part. If  $G_1$ is not cyclic, then
 every formal conjugation $\varphi \in \Difffor$ of $G$ with another (convergent) subgroup $\tilde G \subset \Diff$ is in fact holomorphic, \emph{i.e.}
 \[
 \varphi_* G = \tilde G \implies \varphi \in \Diff \, .
 \]
\end{thm}

\begin{cor}\label{cor:rank2}
Let $G \subset \Diff$ be a solvable subgroup. If the rank of $G_1$ is at least two then $G$ is holomorphically conjugated to a subgroup
of $E_{p,\lambda}$ or $\mathbb A_p$ for some $\lambda \in \mathbb C$ and some $p \in \mathbb N^*$.
\end{cor}

\subsection{Exceptional solvable subgroups}\label{subsec:exceptionalGroups}
The solvable subgroups of $\Diff$ which are not analytically normalizable are described in the next result (see \cite[Proposition 2.10.4]{Frankpseudo}).

\begin{thm}\label{thm:exceptional}
If $G \subset \Diff$ is a solvable subgroup such that $G_1$ has rank one (i.e is infinite cyclic), then
we have one of the following possibilities.
\begin{enumerate}
\item $G$ is abelian and is formally conjugated to the group generated by $ f=\exp(v_{kq,\lambda})$ and $g=\exp(\frac{2\pi i}{q})\exp(\frac{n}{q} v_{kq,\lambda})$, where
$\lambda \in \mathbb C$, $q,k \in \mathbb N^*$ and $n\in \mathbb Z$. In particular, $G_1$ is generated by $f$.
\item $G$ is not abelian and is formally conjugated to the group generated by $f(z)=z/(1-z^p)^{1/p}$ and $g(z)=\exp(2 \pi i/ 2q)z$,
where $p,q \in \mathbb N^*$ satisfies $p=kq$ for some  odd natural number $k$. In this case, the subgroup
generated by $f$ and $g^2$ is an abelian subgroup of index two.
\end{enumerate}
\end{thm}

\subsection{Formally linearizable case}  This corresponds to the situation where $G\subset\Diff$ is formally conjugated to a subgroup of $\mathbb L$. It is well known that there exists examples where this conjugation cannot be made analytic and that this obstruction is related to the presence of small divisors, see for instance \cite{PerezMarco}. Note also that analytic conjugation to a subgroup of $\mathbb L$ holds if and only if there exists a germ of logarithmic  $1$-form $G$ invariant, cf.  \cite[Proposition 3.1.1]{Frankpseudo}.
%Add pointers to the works of Perez-Marco and company.

\section{Abelian holonomy}\label{S:abelianhol}

Most of this section is devoted to  establish  Theorem \ref{THM:Abelian}. The core of its  proof naturally splits in two different cases according
to the  Ueda type of $Y$. The case of infinite Ueda type is treated in Subsection \ref{SS:abelianfiber} while
the case of finite Ueda type is treated in Subsection \ref{SS:abelianfinite}. In Subsection \ref{S:FormallyLinearizableHolonomy}
we  show how to treat the case of foliations with formally linearizable holonomy under additional assumptions.

\subsection{Closed formal meromorphic $1$-forms}
We start by showing that the formal analogue of  Theorem \ref{THM:Abelian} holds true at the formal completion of $X$ along $Y$.

\begin{prop}\label{P:closedformalform}
Let $Y$ be a smooth compact divisor on a projective manifold $X$. Assume that $Y$ is a compact leaf of a foliation $\mathcal F$ on $X$ and that the holonomy of $\mathcal F$ along $Y$ is abelian. The following assertions hold true.
\begin{enumerate}
\item there exists a closed formal meromorphic $1$-form $\hat \omega$ on $Y(\infty)$ defining $\mathcal F_{|Y(\infty)}$;
\item if the holonomy of $\mathcal F$ is formally linearizable, then $\hat{\omega}$ is logarithmic and $\utype(Y)=\infty$;
\item if the holonomy of $\mathcal F$ is not formally linearizable, then $\ord_Y (\hat \omega)_{\infty} -1$ is an integral multiple of $\ord(NY)$.
\end{enumerate}
Moreover, the $1$-form $\hat \omega$ is convergent on a neighborhood of $Y$ whenever the holonomy group
has non unitary linear part in case {\rm (2)}, or its subgroup tangent to the identity is non cyclic in case {\rm (3)}.
\end{prop}
\begin{proof}
(1) The existence of a closed formal meromorphic $1$-form defining $\mathcal F_{|Y(\infty)}$ is explained in Remark \ref{rk:abelianhol}.

(2) If the holonomy is formally linearizable, then the holonomy of $\mathcal F$ is formally conjugated to a subgroup of $\mathbb L$, and consequently $\hat \omega$ is logarithmic. If it is formally  linearizable and unitary, then
$\utype(Y)=\infty$ by definition. If it is formally linearizable, but not unitary, then holonomy group of $\mathcal F$ along $Y$ is holomorphically linearizable, and we obtain a closed logarithmic $1$-form $\omega$ on an Euclidean neighborhood of $Y$. Aiming at a contradiction, assume $\utype(Y) \neq \infty$. Taking general hyperplane sections, we can
reduce to the case where $Y$ is a curve on $X$. If $\utype(Y) \neq \infty$, then Theorem \ref{T:Uedaextension} implies that $\omega$ extends to a closed logarithmic $1$-form on $X$ with polar set contained in $Y$. Since the residue
of $\omega$ along $Y$ is not zero, we arrive at a contradiction with the Residue Theorem.

(3) If $i : (\mathbb C,0) \to X$ is a  germ of transversal to $\mathcal F$ through a point of $Y$, then
$i^* \hat \omega$ is a formal $1$-form  on $(\mathbb C,0)$ invariant by the holonomy of $\mathcal F$. It cannot be logarithmic, since in this case, the holonomy would be linearizable.
The assertion follows from Proposition \ref{P:utypek}.

Finally, the convergence follows from Koenigs Linearization Theorem in case (2), when the holonomy
contains a contraction;  and from Corollary \ref{cor:rank2} in case (3).
\end{proof}

\begin{remark}
Recall from Remark \ref{rk:CounterExample} that (2) does not hold true when $\mathcal F$ is only defined on an analytic neighborhood of $Y$.
\end{remark}
\subsection{Fiber of a fibration}\label{SS:abelianfiber}

\begin{prop}\label{P:infiniteUedatype}
Let $Y$ be smooth compact divisor on a projective manifold $X$ which is a fiber (multiple or not) of a fibration
$f:X \to C$. Assume that $Y$ is a compact leaf of a foliation $\mathcal F$ on $X$, and that the holonomy of $\mathcal F$ along $Y$ is abelian.
Then,  there exists
a projective manifold $Z$, and a generically finite morphism $\pi : Z \to X$, such that $\pi^* \mathcal F$ is defined by a closed rational $1$-form.
\end{prop}
\begin{proof}
Proposition \ref{P:closedformalform} implies the existence of a closed formal  meromorphic $1$-form $\hat \omega$ defining
  $\mathcal F_{|Y(\infty)}$. In particular, $N\mathcal F_{|Y(\infty)} = \mathcal O_{Y(\infty)}(nY)$ where $n$ is the order of poles  of $\hat \omega$ along $Y$. Using the fact that $Y$ is a fiber, we can infer that $N\mathcal F\otimes\mathcal O_{X}(-nY)$ is actually trivial on a Zariski open neighborhood of $Y$. The sheaf $f_*\left(N\mathcal F\otimes\mathcal O_{X}(-nY)\right)$ being torsion free on a curve is a vector bundle (see for instance \cite[Chap.V, Cor.(5.15)]{KobBook}) and has thus locally a nowhere vanishing section if it is non zero. To see that it is indeed the case, it is enough to prove that the formal completion at $f(Y)$ is non zero. But Grothendieck's comparison Theorem \cite[Theorem 4.I.5]{EGA3partie1} implies that the formal completion of $f_* (N \mathcal F\otimes \mathcal O_{X}(-nY)) $ at $f(Y)$ is isomorphic to
$f_*(N \mathcal F\otimes \mathcal O_{Y(\infty)}(-nY))$. As $N \mathcal F\otimes \mathcal O_{Y(\infty)}(-nY)= \mathcal O_{Y(\infty)}$, we finally get
$$\widehat{f_*\left(N \mathcal F\otimes \mathcal O_{X}(-nY)\right)}\simeq f_*\left(N \mathcal F\otimes \mathcal O_{Y(\infty)}(-nY)\right)= f_*(\mathcal{O}_{Y(\infty)})=\widehat{\mathcal{O}_{C,f(Y)}},$$
the fibers of $f$ being connected. Thus there exists a Zariski open neighborhood $U$ of $Y$ where $N\mathcal F_{|U} \simeq \mathcal O_U(nY)$. Consequently, $\mathcal F$ is defined by a rational $1$-form $\omega$ having poles of order $n$ on $Y$ and no other poles on $U$. Therefore, there exists on $Y(\infty)$ a formal holomorphic function $\hat g$ such that
$\hat \omega = \hat g \omega_{|Y(\infty)}$.  Differentiating, and recalling that $d\hat \omega=0$, we obtain that
\[
 d \omega_{|Y(\infty)} = - \frac{d \hat g}{\hat g} \wedge \omega_{|Y(\infty)}\, .
\]
But again, the connectedness of the fibers of $f$ implies that $\hat g = \hat h\circ f$ for some formal holomorphic function $\hat h$ on
a formal neighborhood of $f(Y)$. It follows that for any rational $1$-form $\alpha$ on $C$ without zeros or poles on $f(Y)$,  $d \omega$ and $f^* \alpha \wedge \omega$ are holomorphically proportional on $Y(\infty)$ and consequently rationally proportional on all $X$. More explicitly, there exists a rational function $g$, regular on a neighborhood of $Y$, such that $d \omega = g\cdot f^*\alpha \wedge \omega$. Since $g$ is regular on $Y$, it must be constant along fibers of $f$ and $d ( g\cdot f^*\alpha ) =0$. Therefore $\mathcal F$ is a transversely affine foliation and the result follows from  arguments used in \cite{Gaeljvp} as we will now explain.

Let $\eta = g\cdot f^* \alpha$ and notice that the multi-valued $1$-form $\exp(\int \eta) \omega$ is closed.
If the periods of $\eta$ are not rational multiples of $\pi i$,  then the monodromy of $\exp(\int \eta) \omega$  factors through $C$ \cite[Theorems 4.1 and 5.1]{Gaeljvp}
and $\mathcal F$ must coincide with the fibration. If $\eta$ has poles of order greater than one,  then the restriction of   $\exp(\int \eta) \omega$
to a general fiber of $f$ is exact \cite[proof of Theorem 5.2]{Gaeljvp}. Therefore $\mathcal F$ must coincide with the fibration. At this point, we deduce that, if $\mathcal F$ is not the fibration, then $\eta$ has at worst logarithmic poles, and all its periods
are rational multiples of $\pi i$. It follows that $\mathcal F$ is defined by rational closed one form on a
suitable finite ramified covering of $X$. Taking the resolution of singularities of this covering gives the generically finite morphism $\pi :Z \to X$ such that $\pi^* \mathcal F$ is defined by a closed rational $1$-form. \end{proof}

\begin{remark}\label{rem:disprepencyFibration}We can also prove that $\mathcal F$ is defined by a closed meromorphic $1$-form
at the (analytic) neighborhood of $Y$ by the following alternative argument. Locally along $Y$,
we can define $\mathcal F$ by $\omega_i=dy+y^{k+1}\alpha_i$ with $y$ a (global) reduced analytic equation of $Y$,
$\alpha_i$ holomorphic $1$-form on $U_i$ and $k\in\mathbb Z_{\ge0}$ maximal with these properties. In fact,
$k$ is the contact order of tangency between $\mathcal F$ and the fibration $dy=0$. On $U_i\cap U_j$, we have
$\alpha_i=\alpha_j+g_{ij}\omega_j$ so that ${\alpha_i}_{|Y}={\alpha_j}_{|Y}$ define a global holomorphic
$1$-form $\alpha$ on $Y$. The $1$-form $\alpha$ measures the discrepancy between the foliation and the fibration at the order $k$, the first order for which they differ.

More precisely, if $k=0$, then the holonomy along a loop $\gamma\in\pi_1(Y)$ computed in variable $y$ writes
$$y\mapsto e^{\int_\gamma\alpha}y+\cdots;$$
since all periods of $\alpha$ cannot be imaginary, we deduce that the linear part of the holonomy is not unitary.
This implies that the holonomy group is analytically linearizable and the formal $1$-form $\hat{\omega}$ constructed in Proposition \ref{P:closedformalform} is actually analytic.

Now, when $k>0$, then the holonomy along a loop $\gamma\in\pi_1(Y)$ computed in variable $y$ writes
$$y\mapsto y+\left(\int_\gamma\alpha\right) y^{k+1}+\cdots;$$
since all periods of $\alpha$ cannot be $\mathbb Q$-proportional, we see that the holonomy
cannot be cyclic and the formal $1$-form $\hat{\omega}$ constructed in Proposition \ref{P:closedformalform} is actually analytic.
\end{remark}

\begin{cor}\label{C:finiteUedatype}
Let $Y$ be a smooth compact divisor on a projective manifold $X$.  Assume that $Y$ is a compact leaf of a foliation $\mathcal F$ on $X$, and that the holonomy of $\mathcal F$ along $Y$ is abelian and non formally linearizable.
If $\utype(Y)= \infty$,   then  there exists
a projective manifold $Z$ and generically finite morphism $\pi : Z \to X$ such that $\pi^* \mathcal F$ is defined by a closed rational $1$-form.
\end{cor}
\begin{proof}
According to Proposition \ref{P:closedformalform} (item (3)) the normal bundle of $Y$ is torsion. Since $\utype(Y)=\infty$, Theorem \ref{thm:Neeman} gives a fibration $f : X \to C$ having $Y$ as one of its fibers, and the corollary follows from Proposition \ref{P:infiniteUedatype}.
\end{proof}

\subsection{Finite Ueda type}\label{SS:abelianfinite}

\begin{prop}\label{P:finiteUedatype}
Let $Y$ be smooth compact divisor on a projective manifold $X$. Assume that $Y$ is a compact leaf of a foliation $\mathcal F$ on $X$, and that the holonomy of $\mathcal F$ along $Y$ is abelian. Assume moreover that $\utype(Y) <  \infty$. Then, the holonomy of $\mathcal F$ along $Y$ is not linearizable,
and $\mathcal F$  is defined by a closed rational $1$-form.
\end{prop}
\begin{proof}
By taking general hyperplane sections, we can assume\footnote{It is a consequence of the extension property of integrating factors, see for instance \cite[Theorem 5.1, p.47]{CeMa}. See also Section \ref{sec:extension}.} that $\dim(X)=2$.
Let $G\subset \Diff$ be the image of
the holonomy representation of $Y$, and $G_1$ the subgroup of elements in $G$ with
trivial linear part.

If  $G_1$ is trivial, \emph{i.e.} if the holonomy of $\mathcal F$ along $Y$ is formally linearizable (Theorem \ref{thm:triviallinearpart}), then Proposition \ref{P:closedformalform} item (2) implies that $\utype(Y)=\infty$ contrary to our assumption. This shows that the holonomy of $\mathcal F$ is not linearizable.

If  $G_1$ is not cyclic, then $G$ is analytically normalizable according to
Corollary \ref{cor:rank2} and there exists a closed meromorphic $1$-form defined on a (germ of) transversal
of $\mathcal F$  through a point of $Y$ which is invariant by the holonomy of
$\mathcal F$. Using the local triviality of the foliation, we deduce the existence of
a meromorphic closed $1$-form on a neighborhood of $Y$ defining $\mathcal F$. Then Theorem \ref{T:Uedaextension} allows us to extend the $1$-form to a global meromorphic $1$-form defined on the whole
surface $X$.

If  $G_1$  is infinite cyclic and $G_1$ is not analytically normalizable, the holonomy is then described  up to formal conjugacy by item (1) of Theorem \ref{thm:exceptional}. In particular there exists a formal coordinate $z$ such that the formal $1$-form
\[  \frac{dz}{z^{k+1}} + \lambda \frac{dz}{z}\]
is invariant by holonomy for some suitable choice of $\lambda\in\C$.  The holonomy of the foliation (evaluated along a loop $\gamma$ $\in \pi_1 (Y)$) is thus given by
\[ h_\gamma (z)= a_\gamma \left(z+ m_\gamma z^{k+1} + o (z^{k+1})\right)\] where ${a_\gamma}^k =1$ and $m_\gamma$ is an integer. Equivalently,  we can choose local
formal submersive first integrals $z_i$ for $\mathcal F$ vanishing on $Y$ such that
\begin{equation}\label{e:cyclic}
 z_i= a_{ij} \left( z_j +m_{ij} {z_j}^{k+1} + o ({z_j}^{k+1})\right)
 \end{equation}
with ${a_{ij}}^k=1$ and $m_{ij}\in\mathbb Z$ and such that
 \[\widehat{\omega}= \frac{dz_i}{z_i^{k+1}} + \lambda \frac{dz_i}{z_i}\]
  defines a formal closed meromorphic $1$-form on $Y(\infty)$.
  From Equation (\ref{e:cyclic}), we can derive
\[
-\frac{1}{kz_i^k}+\frac{1}{kz_j^k} =m_{ij} + o({z_j}).
\]

Since $\utype (Y)<\infty$, one can infer from Theorem \ref{thm:Neeman} and item (3) of Proposition \ref{P:closedformalform} that $k=\utype (Y)$ coincides with the order of $NY$.

Let now $y_i$ be a system of local \textit{analytic} equations for $Y$ such that $y_i=z_i\ \mbox{mod}\ {z_i}^{k+2}$. From the previous equality, one can deduce\footnote{Let us note that we are using the same construction as in the proof of Theorem \ref{THM:Existence}, but we use the additional information coming from the assumption on the rank of $G_1$.} that
\[
-\frac{1}{ky_i^k} + \frac{1}{ky_j^k} = m_{ij} + f_{ij}
\]
where $f_{ij}=0$ on $Y$.
Let $g_{ij}= m_{ij} + f_{ij}$  be the (non trivial) cocycle in $H^1(X, \mathcal O_X)$
(away from $Y$, we complete the system of coordinates by $y_i=1$ for instance).
From the exponential sequence in cohomology (and the fact that $X$ is K\"ahler),
it follows that $\exp(2i \pi g_{ij})\in H^1(X, \mathcal O_X^*)$ defines a flat line bundle on $X$:
it admits a unitary connection, \emph{i.e.} there exist local units
$h_i , h_j$ and constants $b_{ij}$ (with $|b_{ij}|=1$) such that
\[
 \exp(2i\pi g_{ij}) =h_i^{-1}h_j b_{ij}.
\]
In restriction to $Y$, we get that $\exp(2i\pi g_{ij})_{|Y}=\exp(2i\pi m_{ij})\equiv1$ is the trivial connection.
By uniqueness of the unitary connection, we must have ${b_{ij}}_{|Y}=1$, and all $h_i$ are constants
in restriction to $Y$. By construction,
\[
\eta = \frac{dy_i}{{y_i}^{k+1}} + \frac{1}{2i\pi}\frac{dh_i}{h_i}
\]
is a well defined closed $1$-form  such that the induced foliation $\mathcal G$ has holonomy along $Y$ equal to the holonomy of $\mathcal F$ up to order  $k+1$ (note that ${dh_i}_{|Y}\equiv0$).

 We want to
compare $\eta$ and $\widehat{\omega}$. Notice that the integration of $\eta$ defines a representation
$\tau_{\eta} : \pi_1(U) \to \mathbb C$ where $U$ is a tubular neighborhood of $Y$ ($C^{\infty}$ neighborhood of course).
Even if $\widehat{\omega}$ is only a formal $1$-form, it  defines a representation $\tau_{\hat \omega} : \pi_1(U-Y) \to \mathbb C$ as follows. Let $\pi : (\tilde U,E) \to (U,Y)$ be the real analytic polar blow-up along $Y$. The exceptional divisor is a $S^1$-bundle over $Y$ with fiber over a point $y\in Y$  parametrizing the rays on
a transversal to $Y$ through $y$.
Take  two adjacent simply connected open subset $V$ and $W$ of $E$.  We can naturally associate to them  two {\it primitives} of $\hat \omega$: those are formal complex functions of the form $-\frac{1}{kz^k} + \lambda \log z$ where  $\log z$ is a branch of the logarithmic meaningful along the rays parametrized by $V$ and $W$. On the intersection $V\cap W$, these two primitives differ by a
constant since their differentials are the same. Following paths on $E$ we obtain a representation $\tau_{\hat \omega} : \pi_1(E) \to \mathbb C$. Since $E$ is a retraction of $U - Y$,  we have the sought representation.

Notice that the  $S^1$-principal bundle $E$  comes endowed with a flat connection with monodromy given by $j^1 \rho : \pi_1(Y) \to \mathbb Z/k\mathbb Z \subset S^1$ the first jet of
the holonomy of $\mathcal F$. If $K = \ker  j^1 \rho$, then we can inject  $K$ into $\pi_1(E)$ by lifting paths in $K$ to flat sections of $E \to Y$.
Let us compare $\tau_{\hat \omega}$ with $\rho_{\mathcal F}$ for $\gamma \in K$.
Along $\gamma$, we see that formal first integrals $z_i$ defined before are simply related by
\[
z_i = z_j ( 1 + m_{ij} z_j ^k + \ldots )
\]
so that $m_{ij} \in \mathbb Z$ can be identified with the image of $\gamma$ under the $(k+1)$-th jet
$$j^{k+1} \rho_\F : K \to (\mathbb C,+) \subset \Difffor$$
of the holonomy of $\mathcal F$ along $Y$. From the equation above we deduce that
\[
-\frac{1}{kz_i^k} + \lambda\log z_i = -\frac{1}{kz_j^k} + \lambda \log z_j +m_{ij}
\]
(recall that we work with a fixed determination of the logarithm which makes sense along $\gamma$ in $K$).
Therefore the restrictions of $j^{k+1}\rho_{\mathcal F}$ and $\tau_{\hat \omega}$ to $K$ coincide. An analogous  relation
holds for the restrictions of $j^{k+1}\rho_{\mathcal G}=j^{k+1}\rho_{\mathcal F}$ and ${\tau_{\eta}}$ to $K$.
Therefore, the period representation $\tau_{\hat \omega - \eta}: \pi_1(E) \to \mathbb C$ of $\hat \omega - \eta$ vanishes on $K$.

Since $y_i=z_i\ \mbox{mod}\ {z_i}^{k+2}$, we deduce from local expressions of $\hat \omega$ and $\eta$   that
$$\hat \omega - \eta=\lambda \frac{dz_i}{z_i}+\text{holomorphic}.$$
Therefore, after integrating, we get a collection of formal reduced equations
$$w_i=\exp\left(2i\pi\int\frac{\hat \omega - \eta}{\lambda}\right)$$
for $Y$, which are well-defined up to multiplication by constants, i.e. a local system.
By construction, the corresponding representation $\pi_1(Y)\to \mathbb{C}^*$
is trivial on $K$, which has finite index in $\pi_1(Y)$, hence it is finite.
We deduce that $w_i=a_{ij}w_j$ along $Y$ with $a_{ij}^k=1$, contradicting that $\utype (Y)<\infty$.
We conclude that this case does not happen, \emph{i.e.} $\lambda=0$ and $\hat \omega - \eta$
has no pole.

We can argue as before that periods of $\hat \omega - \eta$ along
$Y$ vanish on the finite index subgroup $K<\pi_1(Y)$, and therefore on $\pi_1(Y)$.
We can then write $\hat{\omega}-\eta=df$ with $f$ formal along $Y$. As $\utype(Y)<\infty$, we can deduce that $f$ is constant (see Section \ref{ssec:CurvesFiniteType}). Therefore, $\hat{\omega}=\eta$
is a closed rational $1$-form on $X$ defining the foliation $\mathcal F$.
\end{proof}

\begin{remark}
The starting point of the proof of Proposition \ref{P:finiteUedatype} is the existence of a
closed formal meromorphic $1$-form  $\hat \omega$ defining $\mathcal F_{|Y(\infty)}$
and the induced isomorphism  $N\mathcal F_{|Y(\infty)} = \mathcal O_{Y(\infty)}((k+1) Y)$.
A posteriori, we prove that this isomorphism extends to an isomorphism of $N\mathcal F_{|U}$
and $\mathcal O_U((k+1)Y)$ where $U$ is suitable Zariski neighborhood of $Y$ in $X$. If we could
prove that such isomorphism holds true, a priori,  for some Euclidean neighborhood $U$, then it  would be
easy to conclude the proof of Proposition \ref{P:finiteUedatype} since there
would exists (because of the isomorphism) a  meromorphic  $1$-form $\omega$ (convergent but not necessarily
closed) defining $\mathcal F_{|U}$ and with polar divisor equal to $(k+1)Y$. Comparing $\omega_{|Y(\infty)}$ and
$\hat \omega$, we see that they differ by multiplication by a formal holomorphic function. But since $Y$ has finite Ueda type, it is G1 in the sense of  \cite[Section 5]{Hironaka}: this holomorphic function must be constant. We deduce that $\omega_{|Y(\infty)}$
is closed and so is $\omega$.
\end{remark}

\subsection{Proof of Theorem \ref{THM:Abelian}}
If the  holonomy of $\mathcal F$ along $Y$ is formally linearizable, $NY$ has infinite order and $\utype(Y)=\infty$ then
there is nothing else to prove (this is item (2) of the statement).
If  $\utype (Y)=\infty$ and $N_Y$ torsion then  Proposition \ref{P:infiniteUedatype} implies the result.
If instead  $\utype (Y)<\infty$ then  now is Proposition \ref{P:finiteUedatype} that proves Theorem \ref{THM:Abelian}.
The only remaining possibility is that  $\utype (Y)=\infty$, $N_Y$ has infinite order and the holonomy group $G$ along $Y$ is not formally linearizable. In this situation, $G$ is formally conjugated to a subgroup of $\mathbb E_{k,\lambda}$ (see Theorem \ref{T:formalG}) and in particular has finite linear part. This obviously contradicts the fact that $N_Y$ has infinite order. Since these three further cases exhaust all possibilities Theorem \ref{THM:Abelian} follows.
\qed

\subsection{Formally linearizable holonomy}\label{S:FormallyLinearizableHolonomy}

If $\mathcal F$ is a codimension one foliation on a projective manifold $X$ with a compact leaf $Y$ having formally linearizable holonomy,
then the existence of a formal logarithmic $1$-form $\hat{\omega}$ defining $\mathcal F_{|Y(\infty)}$ implies that the bundle $N\mathcal F \otimes \mathcal O_X(-Y)$ has trivial restriction to $Y(\infty)$. It seems reasonable to imagine that this is only possible
because $N\mathcal F \otimes \mathcal O_X(-Y)$ is trivial at an Euclidean neighborhood of $Y$, and even better that it is numerically trivial at a Zariski neighborhood of $Y$. Except when $Y$ is a fiber  a fibration, and we can appeal to Grothendieck's comparison Theorem, we are not aware of results supporting these hopes.
In order to be able to push forward our investigations, in the remaining of this section, we will work under the following hypothesis: $N\mathcal F \otimes \mathcal O_X(-Y)$ is numerically equivalent to a $\mathbb Q$-divisor $D$ whose support is disjoint from $Y$.

\begin{prop}\label{P:abelianunitary}  Let $Y$ be a smooth compact divisor on a compact K\" ahler manifold $X$ with $\utype(Y)=\infty$. Assume that $Y$ is a compact leaf of a foliation $\mathcal F$ on $X$ having formally linearizable holonomy and that $N\mathcal F \otimes \mathcal O_X(-Y)$ is numerically equivalent to a $\mathbb Q$-divisor $D$ with support disjoint from $Y$. Then $\F$ can be defined by a logarithmic closed form, possibly after a finite ramified covering \'etale over $Y$.
In particular, the non analytically linearizable case does not occur under these assumptions.
\end{prop}
\begin{proof}
Assume first that $D$ is a divisor such that $N\mathcal F\otimes \mathcal O_X(-Y)$ is  linearly equivalent to $\mathcal O_X(D)$ and such that $ |D|\cap Y=\emptyset$.
It follows that $\mathcal F$ is defined by a rational $1$-form $\omega$, logarithmic along $Y$, and with poles contained in
$D$. If, as above, we denote by $\hat{\omega}$ the formal logarithmic $1$-form defining $\mathcal F_{|Y(\infty)}$ then we
can write $\omega = \hat g \cdot \hat \omega$ for a suitable section of $\mathcal O_{Y(\infty)}$. If $\hat g$ is not constant
and we denote by $\hat g(Y)$ its value at $Y$ then $\hat g - \hat g(Y)$ is a non-constant formal holomorphic function that vanishes on $Y$.
Therefore $NY$ is torsion and $Y$ is a fiber of a fibration according to Theorem \ref{thm:Neeman}.
We can then apply Proposition \ref{P:infiniteUedatype} to conclude.
If $\hat g$ is constant then the restriction of $d\omega$ to $Y(\infty)$ vanishes identically since $\hat \omega$ is closed.
The identity principle implies that $d\omega$ vanishes identically on $X$.

The general case can be reduced to the case just studied as follows. Let $r$ be the smallest positive integer
such that $rD$ is a divisor (\emph{i.e.} $\mathbb Z$-divisor). The line bundle $N\mathcal F^{\otimes r} \otimes \mathcal O_X(-rY -rD)$
has torsion Chern class and therefore admits a flat unitary connection $\nabla$. When restricted to $Y$ this connection has no
monodromy since $N\mathcal F_{|Y} = {\mathcal O_X(Y)}_{|Y}$ and $|D|\cap Y = \emptyset$ by hypothesis.

If the monodromy of $\nabla$ has order
at least three, then
we can take a finite \'etale covering $\tilde X$ of $X$ in such a way that $\tilde X$ contains at least three pairwise disjoint
hypersurfaces with numerically trivial normal bundles. Hodge index theorem implies that these hypersurfaces have proportional Chern
classes (the pre-images of $Y$) and we can apply \cite{Totaro} (see also \cite{jvpJAG}) to ensure the existence of a fibration having
these hypersurfaces as fibers. It follows that the original $Y$ is itself a fiber of a fibration on $X$ and we can apply Proposition \ref{P:infiniteUedatype}
to conclude.

If the monodromy of $\nabla$ is trivial and $r>1$ then we apply the ramified covering trick (using that $(N\mathcal F \otimes \mathcal O_X(-Y))^{\otimes r} = \mathcal O_X(rD)$) to produce a connected ramified covering  $\pi:\tilde X \to X$ of degree $r$ (since $r$ is minimal) and which does not ramify along $Y$ (since $|D|\cap Y=\emptyset$) such that $\pi^* \mathcal F$  satisfies the assumptions made on the first paragraph of this proof.

Finally if the monodromy of $\nabla$ has order two then we reduce to one of the previous cases by taking an \'etale double covering.
\end{proof}

\section{Solvable holonomy}\label{S:solvablehol}

This section is build around the question below. As in the case of abelian holonomy we have split our study according to the Ueda type of $Y$ and the order of $NY$.

\begin{question}\label{Q:solvable}
If $\mathcal F$ is a foliation on a projective manifold having a compact leaf $Y$ with solvable holonomy, is it true that
$\mathcal F$ is transversely affine ?
\end{question}

\subsection{Fiber of a fibration}
The next statement gives an affirmative answer to Question \ref{Q:solvable} when $Y$ is a fiber of  a fibration.

\begin{prop}\label{solvableholonomy}
Let $Y$ be a smooth compact divisor on a projective manifold $X$ which is a  fiber of some holomorphic fibration $p:X\rightarrow C$ onto a projective curve $C$. Assume also that $Y$ is a compact leaf of a foliation $\mathcal F$ on $X$,  and that the holonomy of $\mathcal F$ along $Y$ is solvable.
Then $\F$ is a transversely affine foliation
K\ '\end{prop}
\begin{proof}
For the sake of simplicity, we will suppose that $Y$ is a regular fiber, i.e $NY\simeq{\mathcal O}_Y$. Indeed, every case can be reduced to this latter by an appropriate finite base change; on the other hand,
 a foliation is transversely affine if, and only if, its pull-back by a dominant morphism is (\cite[Theorem 1.4]{Casale}, \cite[Proposition 2.9]{Gaeljvp}).
Let $U \subset C$ be the open subset of regular values of $p$, and let $u_1: U_1 \to U$ be the universal covering of $U$. If we set $V= p^{-1}(U)$, then
we have the  diagram
$$ \xymatrix{
  Y \ar[r] &  X \ar[d]^p & \ar[l] V  \ar[d]^{p}  &  V_1  \ar[l]_{\pi} \ar[d]^{p_1} \\
  & C &  U \ar[l] & \ar[l]_{u_1}  U_1 \\
}$$

 We will assume  that the holonomy is solvable but non abelian, since the abelian case has been settled in Proposition \ref{P:infiniteUedatype}. In particular, the linear part of the holonomy group is analytically equivalent to a subgroup of $\mathbb A_k$ for a suitable $k\in \mathbb N^*$. Indeed, being solvable non abelian, the linear part of the monodromy cannot be trivial (see Section \ref{S:formalclassification})
and we can argue as in Remark \ref{rem:disprepencyFibration} to deduce that the linear part is not
unitary: we are not in the exceptional case of Section \ref{subsec:exceptionalGroups}
and we can conclude by Corollary \ref{cor:rank2}.
Therefore, there exists an Euclidean neighborhood $W$ of $Y$ in $X$, a holomorphic flat connection $\nabla_W$ on $\left(N \mathcal F\otimes \mathcal O_X(-(k+1)Y)\right)_{\vert W}$ and a meromorphic $1$-form $\omega$ defined on $W$ with coefficients in $N \mathcal F \otimes \mathcal O_X(-(k+1)Y)$ and polar divisor equal to $(k+1)Y$
 which defines $\mathcal F$
  such that $\nabla_W(\omega_{|W})=0$.

We will first prove that the monodromy representation of this transversely affine structure
for $\mathcal F_{|W}$ extends to a representation of the fundamental group of a Zariski neighborhood of $Y$.

Let $\mathcal E$ be the rank one local system on $W$ defined by the flat sections of $\nabla_W$. If we consider its pullback $\pi^* \mathcal E$ to $V_1$
then the simple connectedness of $U_1$ allows us to identify $\pi_1(V_1)$ with $\pi_1(Y)$ and, as a by product, to extend $\pi^* \mathcal E$ to a rank one local system $\mathcal D$ defined on $V_1$. Let $D = \mathcal D \otimes \mathcal O_{V_1}$ be the associated line bundle.

We claim that $D \simeq \pi^* ( N\F \otimes \mathcal O_X(-(k+1)Y) )\simeq \pi^* N\mathcal F$. Indeed, let $\mbox{Jac} (p_1)\rightarrow U_1$ the relative Jacobian of $p_1:V_1 \to U_1$. Since
$D$ and $\pi^* N\F $ are both flat on the fibers, they both induce holomorphic sections
$$s_D, s_{\pi^* N\mathcal F }: U_1\rightarrow \mbox{Jac} (p_1).$$
Actually, these section are the same as they coincide on some neighborhood of $p_1(\pi^{-1} (Y))$.  One can then conclude that $D$ and $\pi^* N\mathcal F$ are equal on restriction to fibers (up to isomorphism). By  triviality  of $H^1(U_1, \mathcal{O}^*_{U_1})$, we infer that  equality holds on the whole $V_1$.
Similarly, we deduce that $D \simeq \pi^* ( N \F \otimes \mathcal O_X(-(k+1)Y))$.

Let ${\nabla }_D$ be the  flat holomorphic connection on  $D$ determined by  $\mathcal D$.
Set $\Omega=\pi^* (\omega)$. Modulo the previous identification  of line bundles, $\ {\nabla }_D \Omega$
  is well defined as a meromorphic section of $\Omega_{V_1}^2\otimes D$.

We are now going to compare the connection $\nabla_W$  (a priori only defined on a neighborhood $W$ of $Y$) with
$\nabla_D$. Since $\nabla_W ( \omega_{|W})=0$, we have that $\pi^* \nabla_W \Omega=0$. On the other hand, as $\pi^*\nabla_W$ has the same monodromy as $\nabla_D$   on $W_1=\pi^{-1} (W)$,  it follows that  $\pi^*\nabla_W - \nabla_D=p_1^*\eta$ where $\eta$ is a meromorphic form on     $p_1(W_1)$. But, since $\pi^ * \nabla_W (\Omega ) =0$, we also have that
\begin{equation}\label{e:transaffine}
\nabla_D(\Omega) = p_1^* \eta \wedge \Omega \, .
\end{equation}
Moreover, $p_1^* \eta$ can be expressed as $A  p_1^*{ \eta}'$ where ${\eta}'$ is a well defined meromorphic form on $U_1$ (recall that $U_1$ is nothing but the unit disc) and $A$ is a meromorphic function defined on $W_1$. By re-injecting $A  p_1^*{ \eta}' =p_1^* \eta$ in  Equation (\ref{e:transaffine}), one observes that $A$, hence $p_1^* \eta$, extends as a global meromorphic object on $V_1$ ($\nabla_D(\Omega)$ and $\Omega$ being both globally defined).
This shows that the connection $\pi^* \nabla_W$ extends to a flat meromorphic connection $\nabla_1$ on $V_1$.

Let $b=p(Y)$, $a$ and $a'$ be two points in $u_1^{-1} (b)$ and ${\nabla_1}_a$, ${\nabla_1}_ {a'}$ be the germs of $\nabla_1$ along $F_a$ and $F_{a'}$. An easy calculation gives
$$\pi_*{\nabla_1}_a-\pi_*{\nabla_1}_ {a' }=\xi + p^*\theta$$
where $\xi$ a meromorphic closed one form tangent to $\F$  defined in a neighborhood of $Y$ and $\theta$ is a meromorphic form defined near $p(Y)$. The non abelianity of the holonomy group forces $\xi$ to vanish identically.  This clearly implies that the monodromies associated to both connections $\pi_*{\nabla_1}_a$ and $\pi_*{\nabla_1}_ {a' }$ are the same (near $Y$). But this implies that the connection $\nabla_1$ descends to a connection
over the Zariski open neighborhood $U$ of $Y$; therefore, the transversely affine structure of $\mathcal F$ originally defined only at an Euclidean neighborhood $W$ of $Y$, is now defined on $U$. In particular, the monodromy of the transversely affine structure of $\mathcal F_{|W}$  extends to a representation
$\rho : \pi_1(U) \to \Aff(\mathbb C)$ of the fundamental group
of the Zariski open neighborhood $U$ of $Y$.

Since $\rho$ is not virtually abelian (remind its linear part is not unitary), we deduce from \cite{Bartolo:arXiv1005.4761} the existence of a rational map  $f : X \dashrightarrow B$ from $X$ to a curve $B$ which is regular at a neighborhood of $Y$ (\emph{i.e.} an actual morphism) and   factors $\rho$. To wit, there exists a representation $\tilde \rho : \pi_1^{orb}(B) \to \Aff(\mathbb C)$ such that $\rho = \tilde \rho \circ f_*$.
This allows us to construct a global connection $\nabla$ on the trivial line bundle over $U$,
having the same monodromy as $\nabla_W$ when restricted to $W$.  Indeed, $\nabla$ is of the form $d + f^* \eta$ where $\eta$ is a holomorphic $1$-form on $B$.
This implies that the normal bundle of $\mathcal F$ is trivial when restricted to a neighborhood of $Y$, and globally
can be written as the line bundle associated to a divisor with irreducible components
contained in fibers of $p$. Therefore, there exists a rational $1$-form $\tilde \omega$
defining $\mathcal F$ with divisor of zeros and poles contained in fibers of $p$. Furthermore, we can assume that, at a neighborhood of $Y$, the polar divisor of $\tilde \omega$ is $(k+1)Y$.

Since $\nabla$ and $\nabla_W$ have the same monodromy on $W$, they differ by an exact $1$-form $dH$.
On $W$, we also have that $\tilde \omega =g \omega_{|W}$ for some holomorphic function $g$ defined on $W$. Hence, on $W$, we can write
\[
\nabla(\tilde \omega) = \nabla_W(g \omega_{|W}) - dH\wedge \tilde  \omega = \left( \frac{dg}{g} - dH \right) \wedge \tilde \omega .
\]
We thus see that $\frac{dg}{g} - dH$ is the restriction to $W$ of a closed rational $1$-form which can be written as  $p^* \alpha$ for a suitable rational
$1$-form $\alpha$ defined on $C$.
We have thus established  the identity
\[
\left(\nabla - p^*\alpha\right)  \tilde \omega = 0 \, ,
\]
showing that $\mathcal F$ is transversely affine, concluding the proof of the proposition.
\end{proof}

\subsection{Finite Ueda type}\label{sec:finiteuedatype}
In the proof of the next proposition, we will make use of  Atiyah's interpretation for a holomorphic connection on a  locally free sheaf \cite{Atiyah} which we now proceed to recall.
If $\mathcal E$ is a locally free sheaf, then we
define another locally free sheaf $D(\mathcal E)$ as follows. As a sheaf of $\mathbb C$-modules $D(\mathcal E)$ is  $ \mathcal E \oplus \Omega^1_X \otimes \mathcal E$, but the structure of $\mathcal O_X$-module on $D(\mathcal E)$ is not usual one. Multiplication by an element $f \in \mathcal O_X$ is given by
\[
f \cdot  ( s , \sigma ) = ( f s , df \otimes s + f \sigma ) .
\]
The sheaf $D(\mathcal E)$ fits into the natural exact sequence
\[
0 \to  \Omega^1_X \otimes \mathcal E \to D(\mathcal E) \to \mathcal E \to 0 \, .
\]	
Atiyah proved that holomorphic connections on $\mathcal E$ are in bijection with splittings
$\varphi: \mathcal E \to D(\mathcal E)$ of this exact sequence. In particular, given a holomorphic connection $\nabla$ on $\mathcal E$, we obtain a section of $\mathcal E^ * \otimes D(\mathcal E)$ which maps to the identity in $\mathcal E^ * \otimes \mathcal E$
through the morphism induced by the exact sequence above.

\begin{prop}\label{P:solvablefinitetype}
Let $\mathcal F$ be a foliation on a projective manifold $X$. Assume that $\mathcal F$ has a compact leaf $Y$ with solvable holonomy, and that the Ueda type of $Y$ is finite. Then, one of the following assertions holds true:
\begin{enumerate}
		\item $\mathcal F$ is a transversely affine foliation.
		\item The holonomy of $Y$ is virtually abelian, and there exists a positive integer $q$ such that  the Ueda type of $Y$ is at least $q$ and the normal bundle of $Y$ is torsion of order $2q$.
\end{enumerate}
\end{prop}
\begin{proof}
There is no loss of generality in assuming that the holonomy is not abelian since the abelian case is covered by Proposition \ref{P:finiteUedatype}. Also, we can assume that $X$ is a surface since a foliation on a projective manifold is transversely affine if, and only if, a general hyperplane section of it is transversely affine, see Section \ref{sec:extension}.
	
Assume first that the holonomy of $\mathcal F$ along  $Y$ is analytically normalizable. In this case, we can construct a transversely affine structure for $\mathcal F$ at an Euclidean neighborhood $U$ of $Y$, \emph{i.e.} we can construct a flat meromorphic connection $\nabla_U$ on $N\mathcal F_{|U}$ having  polar divisor supported  on  $Y$ such that $\nabla_U( \omega_{|U})=0$ where $\omega \in H^0(X,\Omega^1_X\otimes N \mathcal F)$ is a twisted $1$-form defining $\mathcal F$. From the discussion on Atiyah's interpretation of connections, we
obtain a meromorphic section $\sigma$ of $(N \mathcal F^ * \otimes D(N\mathcal F))_{|U}$ over
$U$ inducing $\nabla_U$. Ueda's Theorem \ref{T:Uedaextension} allows us to extend $\sigma$ to a rational section of $N\mathcal F^* \otimes D(N\mathcal F)$. Equivalently, we are able to extend the connection $\nabla_U$ to a rational connection $\nabla$ over all $X$. Since flatness is a closed condition, $\nabla$ is flat and, similarly,  $\nabla(\omega)$ vanishes identically. This proves that $\mathcal F$ is transversely affine.

Assume now that the holonomy group $G$ of $\mathcal F$ along $Y$ is not analytically normalizable. Since we
are assuming that $G$ is not abelian,
we have that there exists positive integers $p,q$ and an odd integer $k$ satisfying $p=kq$
such that $G$ is formally conjugate to the group generated by
\[
	 f(z) = \frac{z}{(1-z^p)^{1/p}} \quad \text{ and } \quad g(z) = \exp(\frac{2\pi i}{2q}) z \, ,
\]
see Theorem \ref{thm:exceptional} item (2). It follows that $\utype(Y)\ge p=kq$ and $\ord(NY)=2q$.
If $k>2$ then Theorem \ref{thm:Neeman} implies that $\utype(Y)=\infty$ and $Y$ is a fiber of fibration contrary to our assumptions.
Since $k$ must be an odd integer according to Theorem \ref{thm:exceptional}, we conclude that  $k=1$ and $p=q$ as stated.
\end{proof}

\subsection{Virtually abelian holonomy} \label{SS:virtuallyabelian}
As in the case of formally linearizable holonomy, we do not know how to deal with the case of non-analytically normalizable
virtually abelian holonomy. Anyway, in that case we have that
\[
N\mathcal F_{|Y(\infty)} = \mathcal O_{Y(\infty)}((q+1)Y) \otimes T
\]
where $T$ is a flat line bundle of order two. If we assume that a global version of this identity holds, then
we are able to reduce to the case of abelian holonomy treated by Theorem \ref{THM:Abelian}.

\begin{prop}\label{P:Virtuallyabelian}
Let $Y$ be a compact leaf of a codimension one foliation $\mathcal F$ on a projective manifold $X$.
If $\utype(Y)<\infty$,  $\mathcal F$ is as in item (2) of Proposition \ref{P:solvablefinitetype}, and
there exists a $\mathbb Q$-divisor $D$ disjoint from $Y$ such that $N\mathcal F$ is numerically equivalent to
$(q+1)Y + D$, then there exists a generically finite morphism $\pi: \tilde X \to X$,
which is \'etale over a neighborhood of $Y$, and such that
the holonomy of $\pi^* \mathcal F$ along  $\pi^*(Y)$ is abelian. In particular, the foliation $\pi^* \mathcal F$
satisfies conclusion (1) of Theorem \ref{THM:Abelian}.
applies
\end{prop}
\begin{proof}
If $D$ is a $\mathbb Z$-divisor then $\mathcal L = N \mathcal F \otimes \mathcal O_X(-(q+1)Y - D)$ belongs to
$\Pict(X)$. The restriction of $\mathcal L$ to $Y$ coincides with $NY^{\otimes -q}$ which is a torsion
line bundle of order two according to our hypothesis. Since $\utype(Y)<\infty$, it follows from Proposition \ref{P:Albanese}
that $\mathcal L$ itself is
a torsion line bundle. The conclusion follows by taking the associated covering.

If $D$ is not a $\mathbb Z$-divisor then let $r$ be the smallest positive integer such that $rD$ is. The line bundle
$\mathcal L = N\mathcal F^{\otimes r} \otimes O_X(-r(q+1)Y - rD)$ is torsion. Thus, after taking the associated covering,
we can assume that it is trivial. Applying the ramified covering trick, we produce a finite covering ramified only over the support of $D$,
and such that the pull-back of our foliation satisfies the assumption of the first paragraph of this proof.
\end{proof}

\section{Factorization}\label{S:Factor}

For a fixed $k\ge1$, the  group $J^k \Difffor$ of $k$-th jets of formal diffeomorphisms of $(\mathbb C,0)$ is a solvable linear algebraic group. In \cite{Campana99,Campana01} (resp. \cite{Brudnyi03a,Brudnyi03b}), representations  of K\"{a}hler groups
on  solvable groups (solvable matrix groups)  are studied. In view of the results obtained in these articles, it is natural to ask if some factorization results hold true for arbitrary non virtually abelian representations on $\Difffor$. As recalled in Section \ref{S:subgroups}, Theorem \ref{thm:center},
 the center  $Z(G)$ of a non virtually abelian $G<\Difffor$ is necessarily finite (thus justifying the first assertion of Theorem \ref{THM:Holfactorization}), hence conjugated to a group of unit roots.  We split now the study according to the order of tangency of the given representation.

\subsection{Representations with trivial linear part.}
We consider  a compact K\"ahler manifold $Y$ and a representation
$$\rho : \pi_1(Y) \to \Difffor_\nu$$
where $\nu$ is a positive integer and $\Difffor_\nu$ is the subgroup\footnote{By convention $\Difffor_0:=\Difffor$.} of $\Difffor$ whose elements are tangent to identity to order $\geq \nu$ (see Definition \ref{def:filtration-diff}). We assume moreover that $\nu$ is the greatest integer having this property. Let $J^{k} \rho$ the induced representation on $k$-jets. In particular, $\nu$ is the first positive integer such that $J^{\nu +1} \rho$ is not trivial. The first lemma shows that factorization of the full representation is equivalent to factorization of a finite truncation.

\begin{lemma}\label{L:factorization_k-jets}
Assume that, for some $k\geq \nu+1$, $J^k \rho$ factors through a curve $C$ (via a morphism $Y\to C$), then $J^{k+\nu+1}\rho$ factors through $C$. In particular, by induction, $\rho$ factors through $C$.
\end{lemma}
\begin{proof}
By assumption, we have a fibration $f:Y\to C$ such that $J^k \rho$ factors through $C$. Let $U$ be
a dense Zariski open subset of $C$ over which $f$ is a smooth fibration. Let $V=f^{-1} (U)$ and $F\subset V$ be a smooth fiber of $f$. The monodromy representation will be denoted by
$$\mu: \pi_1(U)\rightarrow \mbox{GL} (H_1(F,\C)).$$
The jet filtration on $\Difffor$ provides us with the following exact sequence:
$$0\to {\C}^{\nu+1} \to  J^{k+\nu+1} \Difffor_\nu \to  J^{k} \Difffor_\nu \to 0.$$
Note that we have a non trivial natural action of  $J^{k} \Difffor_\nu$ onto $\C^{\nu+1}$ induced by conjugation in $J^{k+\nu+1}\Difffor$, namely
\begin{equation}\label{e:actiononthefiberrepresentation1}
g\cdot (b_{k+1},...,b_{k+\nu+1})=(b_{k+1},...,b_{k+\nu},b_{k+\nu+1}+(\nu-k)a_{\nu+1}b_k)
\end{equation}
where $g(x)= x+...+a_kx^k\ \mathrm{mod}\ (x^{k+1})$. By hypothesis, the truncated representation $J^{k+\nu+1}\rho$ induces by restriction a representation $\phi: \pi_1(F)\rightarrow \C^{\nu+1}$ which factorizes through $\varphi :H_1(F)\rightarrow \C^{\nu+1}$. Set $H=\varphi (H_1(F))$ and $H_\C= H\otimes\C$. Remark that $\pi_1(U)$ acts on $H$ by multiplication as defined in Equation (\ref{e:actiononthefiberrepresentation1}), and denote by $q$ this action. Define $G$ as the subgroup of $\mbox{GL}(H_1(F,\C))$ which preserves the kernel of the morphism ${\varphi}_\C:H_1(F,\C)\rightarrow H_\C$ induced by $\varphi$ and let $\beta:G\rightarrow\mbox{Aut}(H_\C)$ the canonical surjection.  Because the action of $\pi_1(U)$ on $H_1(F)$ is inherited from the action of $\pi_1(X)$ onto itself by conjugation, one obtains the following commutative diagram:
 $$\xymatrix{
    G \ar[r]^\beta  & \mbox{Aut} (H_\C) \\
    \pi_1(U) \ar[u]_\mu\ar[ru]^q &
  }.$$
Let $Z$ be the Zariski closure of $\mu(\pi_1 (U))$. According to Deligne's semi-simplicity theorem \cite{Deligne}, the identity component $Z_0\subset Z$ is semi-simple On the other hand, by Equation (\ref{e:actiononthefiberrepresentation1}), $\beta(Z_0)$ is infinite abelian, which leads to a contradiction unless $H_\C=\{0\}$ and consequently $J^{k+\nu+1}\rho$ factorizes through $f$.
\end{proof}
\begin{remark}
This kind of factorization results are probably well known by specialists working on representation of K\"ahler groups. However, we didn't manage to extract a precise statement in the literature. The use of Deligne's theorem in the proof of Lemma \ref{L:factorization_k-jets} above is due to Campana (\cite{Campana01}, proof of Theorem 4.1, p.619) and  our argumentation follows the same line than {\it loc.cit}.
\end{remark}

The next proposition is purely group theoretic and is in a way reminiscent from the proof of Theorem \ref{T:higher dimensional formal foliation}.
\begin{prop}\label{P:non_abelian_Diff}
	If $\Gamma$ is a non abelian subgroup of $\Difffor_1$, then there exist two classes of $a,\,b\in H^1(\Gamma,\C)$ which are not proportional and such that $a\wedge b=0$ in $H^2(\Gamma,\C)$.
\end{prop}
\begin{proof}
	To simplify, let us first assume that $\Gamma$ is not contained in $\Difffor_2$. It implies the following: if $g(z)=z+a(g)z^2+\dots$ then the morphism $a:\Gamma\to \C$ is not zero. Now, consider the following expression:
	\begin{align*}
		F_0(g)&:=\frac{1}{g(z)}-\frac{1}{z}=\frac{1}{z+a(g)z^2+\dots}-\frac{1}{z}\\
		&=-a(g)+b(g)z+c(g)z^2+\dots
	\end{align*}
	Using the fact the $F_0$ satisfies the obvious cocycle relation $F_0(gh)=F_0(g)\circ h+F_0(h)$, we infer the following equality:
	\begin{align*}
		a(gh)+b(gh)z+&c(gh)z^2+\dots=-a(g)+b(g)(z+a(h)z^2+\dots)\\
		&+c(g)(z+a(h)z^2+\dots)^2+\dots-a(h)+b(h)z+c(h)z^2+\dots
	\end{align*}
	Identifying the coefficients of $z$ and $z^2$, we get: $b(gh)=b(g)+b(h)$ so that $b\in H^1(\Gamma,\C)$ and $c(gh)=c(g)+c(h)+b(g)a(h)$. This last identity exactly amounts to saying that $a\wedge b=0$ in $H^2(\Gamma,\C)$. If $b$ is not proportional to $a$, we are done. We can thus assume that $b=\lambda a$ and consider the following:
	$$F(g)=F_0(g)-\lambda\log\left(\frac{g(z)}{z}\right).$$
	It satisfies the same cocycle relation ($F(gh)=F(g)\circ h+F(h)$) and it has the following expansion:
	$$F(g)=-a(g)+\lambda a(g)z+\dots-\lambda\log(1+a(g)z+\dots)=-a(g)+a_2(g)z^2+\dots$$
	Assume from now on that there exists a coordinate in which $F$ can  be written:
	$$F(g)=-a(g)+a_k(g)z^k+a_{k+1}(g)z^{k+1}\dots$$
	for some $k\ge2$ and some functions $(a_j)_{j\ge k}$. Using the cocycle relation, we see that
	\begin{align*}
		-a(gh)&+a_k(gh)z^k+a_{k+1}(gh)z^{k+1}+\dots=-a(g)+a_k(g)(z+a(h)z^2+\dots)^k+\\
		&a_{k+1}(g)(z+a(h)z^2+\dots)^{k+1}+\dots-a(h)+a_k(h)z^k+a_{k+1}(h)z^{k+1}+\dots
	\end{align*}
	We still identify coefficients and get: $a_k(gh)=a_k(g)+a_k(h)$ and $a_{k+1}(gh)=a_{k+1}(g)+a_{k+1}(h)+ka_k(g)a(h)$. This means exactly that $a_k$ is a class in $H^1(\Gamma,\C)$ such that $a_k\wedge a=0$. If $a_k=\lambda_k a$ we can (exactly as in the proof of Theorem \ref{T:higher dimensional formal foliation}) perform the change of coordinate
	$$y=z+\frac{\lambda_k}{k-1}z^{k+1}.$$
	Expanding $F(g)$ with respect to $y$ we get:
	$$F(g)=-a(g)+a'_{k+1}(g)y^{k+1}+\dots$$
	If we can go on this procedure inductively, we end up with a formal coordinates (still denoted by $z$) such that $F(g)=-a(g)$. This is equivalent to saying that any $g\in \Gamma$ preserves a rational formal 1-form expressed as
	$$\frac{dz}{z^{2}}-\lambda\frac{dz}{z}.$$
	According to the interpretation of Section \ref{S:interpretation}, we conclude that $\Gamma$ is abelian. If it is not the case, the process above has to stop at some point, and it gives a class $a_k$ that is not proportional to $a$ and such that $a_k\wedge a=0$.
	
	In the general case, if $\Gamma$ is contained in $\Difffor_\nu$ but not in $\Difffor_{\nu+1}$, we have to modify the expression of $F$:
	$$F(g)=\frac{1}{g(z)^\nu}-\frac{1}{z^\nu}+\lambda\log\left(\frac{g(z)}{z}\right).$$
\end{proof}

We also recall a variation of the Castelnuovo-De Franchis theorem due to Catanese \cite[Theorem 1.10]{Cat}.

\begin{lemma}[Catanese]\label{L:Castelnuovo-deFranchis-Touzet}
	Let $Y$ be a compact K\"ahler manifold and $\alpha,\beta\in H^1 (Y,\C)$ such that $\alpha\cup \beta=0$. Then
	\begin{enumerate}
		\item either $\alpha$ and $\beta$ are colinear,
		\item or there exists a morphism $f: Y\rightarrow C_g$ with connected fibers onto a curve of genus $\geq 2$ and ${\alpha}', {\beta}'\in H^1(C_g, \C)$ such that $\alpha=f^* {\alpha}'$ and $\beta=f^*{\beta}'$.
	\end{enumerate}
\end{lemma}

For the case of representations tangent to identity and thanks to Remark \ref{r:tangenttoid}, Theorem \ref{THM:Holfactorization} follows from

\begin{thm}\label{T:factorization_tangent_id}
Let $\rho : \pi_1(Y) \longrightarrow \Difffor_\nu$ be a representation where $\nu\ge1$. Then
\begin{enumerate}
\item either $\rho$ has abelian image,
\item or $\rho$ factors through a curve.
\end{enumerate}
\end{thm}
\begin{proof}
Lemma \ref{L:factorization_k-jets} reduces  the proof of this  result to  Lemma \ref{L:Castelnuovo-deFranchis-Touzet} and Proposition \ref{P:non_abelian_Diff}.
\end{proof}

\subsection{Representation with finite but non trivial linear part.} \label{finitelinearpart} 
We assume here that the image of $\rho$ is non virtually abelian with finite linear part, \emph{i.e.} $\mbox{Im}\ J^1\rho<\infty$. Let $\pi:Y'\to Y$ be the finite \'etale Galois cover determined by $\mbox{Ker}\ J^1\rho$. From the previous analysis, the pull-back representation ${\pi}^* \rho$ is tangent to identity and factors through a curve $C$. If $F$ denotes a fiber over $C$, ${\pi}^* \rho$ is trivial in restriction to $\alpha (F)$ for any deck transformation $\alpha$. On the other hand, ${\pi}^* \rho$ has infinite image. This implies that $\alpha$ preserves the fibration. Projecting the fibers onto $Y$  and taking if necessary Stein factorization, we obtains a surjective morphism $p:Y\to C'$ from $Y$ to a curve $C'$ with connected fibers along which $\rho$ has finite image. Taking the exact sequence associated to this fibration (up to shrinking the base), we get that the image of $\rho_{|F}$, $F$ a generic fiber, lies in the center\footnote{A priori, this only establishes that $\mbox{Im}\ \rho_{|F}$ is normal in $\mbox{Im}\ \rho$.  In order to justify $\mbox{Im}\ \rho_{|F}\subset Z(\mbox{Im}\ \rho)$, simply note that  $\mbox{Im}\  \rho_{|F}$ is finite, hence lies in a finite subgroup of $\mathbb L$ (Corollary \ref{cor:finiteorder}) and that the linear part is preserved under conjugacy.} of $\mbox{Im}\ \rho$. This proves Theorem \ref{THM:Holfactorization} for representations with finite linear part. \qed

\subsection{Representation with infinite linear part.}We assume here that the image of $\rho$ is non abelian with infinite linear part: $\mbox{Im}\ J^1\rho=\infty$. Let $\pi:Y'\to Y$ a finite \'etale Galois cover such that $\mbox{Im}\ J^1 \rho$ is torsion free. We begin by a result analogous to Lemma \ref{L:factorization_k-jets}.
\begin{lemma}
Assume that, for some $k\geq 1$, $J^k \rho$ factors through a curve $C$, then $J^{k+1}\rho$ factors through $C$. In particular, by induction, $\rho$ factors through $C$.
\end{lemma}
\begin{proof} Once again, the proof resorts to Deligne's semi-simplicity theorem. Indeed, we conclude similarly to the proof of Lemma \ref{L:factorization_k-jets} observing here that the jet filtration on $\Difffor$ provides us with the following exact sequence
\[
0\to \C \to  J^{k+1} \Difffor \to  J^{k} \Difffor \to 0.
\]
The natural action of  $J^{k} \Difffor$ onto $\C$ induced by conjugation in $J^{k+1}\Difffor$ is then defined by
\begin{equation}\label{actiononthefiberrepresentation}
g\cdot a=\lambda^{-k} a
\end{equation}
 where $g(x)=\lambda x+...+a_kx^k\, \mathrm{mod}\, (x^{k+1})$. The end of the proof is then parallel to the one of Lemma \ref{L:factorization_k-jets}.
\end{proof}
 Let $m\geq 2$ be the first positive number such that $J^m\rho$ has a non abelian image. In this context, this is equivalent to say that, for every $\gamma\in\pi_1 (Y)$, $J^m  \rho (\gamma)= \lambda_\gamma z +a_\gamma z^m$ with $\gamma\to a_\gamma$ a non trivial map. In particular, ${J^1 \rho}^{\otimes {m-1}}$ possesses a nontrivial affine extension, i.e $H^1(Y,{J^1 \rho}^{\otimes {m-1}})\not=0$. According to \cite{Campana01},  there exists a Galois finite \'etale cover $\pi:Y'\to Y$, a surjective morphism $f'$ from $Y'$ to a curve $C'$ through which $\pi^*{J^1\rho}^{\otimes{m-1}}$ factors.

Let us choose $\pi$ such that  $\pi^* {J^1\rho}$ factors through $f'$. According to the previous lemma, the whole representation $\pi^*\rho$ factors through $f'$. Arguing as in Section \ref{finitelinearpart}, $f$ projects to a morphism $f:Y\to C$ which after Stein factorization provides the factorization given in Theorem \ref{THM:Holfactorization} and concludes its proof.

\begin{remark}\label{r:infinitecase} For infinite linear part, Theorem \ref{THM:Holfactorization} has been established under the sole assumption of non abelianity of the image of $\rho$. Actually, it is not difficult to see in this setting that "non abelian" is equivalent to "non virtually abelian".
\end{remark}

%\begin{remark}\label{ex:finite_center}
%If the linear part is non trivial, we cannot hope for a factorization result without passing to a finite \'etale cover of $Y$. Consider for instance $Y=C\times E$ where $C$ is an %hyperbolic curve with a representation $\rho_C:\pi_1(C)\to \Difffor$ whose image commutes to $z\mapsto \pm z$ and $E$ is an elliptic curve with $\rho_E$ a representation in $\Difffor$ %whose image is exactly $\langle z\mapsto \pm z\rangle$. If $\rho=\rho_C\times \rho_E$ is the product representation, then it is obvious that $\rho$ factor through $C$ only after taking the %double cover of $Y$ corresponding to the kernel of $\rho_E$.
%\end{remark}

\subsection{Factorization of foliations. Proof of Theorem \ref{THM:Factorization}}
Let $X$ be a compact K\"{a}hler manifold of dimension at least $3$, $\mathcal F$ be codimension one foliation
on $X$ and $Y \subset X$ a compact leaf of $\mathcal F$ such that $NY$ has order $m$.

Suppose first that $\utype (Y)>m$, \emph{i.e.} $Y$ is a fiber of a fibration $f:X\to C$ over a curve.
If $\mathcal F$ coincides
with the fibration, then there is nothing else to prove; from now on, we will assume that $\mathcal F$ and the fibration
are distinct foliations.

The normal bundle of $\mathcal F$ restricted to $Y$ coincides with the normal bundle of $Y$ and is therefore torsion.  In particular, $N\mathcal F_{|Y}$  has zero real Chern class. If we restrict $N\mathcal F$ to $Y_{t} = f^{-1}(t)$
for a general $t \in C$ then it is perhaps not true that $N\mathcal F_{|Y_t}$ is still torsion, but certainly the real Chern class
of $N\mathcal F_{|Y_t}$ is zero. Two possibilities can occur: (a) for a general $t \in C$, $N\mathcal F_{|Y_t}$ is not torsion; or (b)
 $N\mathcal F$ is torsion on a Zariski neighborhood of $Y$.

Let us consider first  case (a). Let $\omega \in H^0(X,\Omega^1_X \otimes N\mathcal F)$ be a
twisted $1$-form defining $\mathcal F$, and let $i_t : Y_t \to X$ be the inclusion. Since $\mathcal F$
is distinct from the fibration $f$, the pull-back $i_t^* \omega \in H^0(Y_t, \Omega^1_{Y_t} \otimes N\mathcal F_{|Y_t})$
is non-zero for a general $t$. Also, by assumption, $N\mathcal F_{|Y_t}$ has zero Chern class but it is not torsion.
Therefore, according to \cite{Campana01}, there exists a morphism $g_t:Y_t \to C_t$ to a curve such that $i_t^* \omega$
is the pull-back of a twisted $1$-form on $C_t$. In particular, the leaves of $\mathcal F_{|Y_t}$ are the fibers of $g_t$.
Since $t$ is general, we obtain through a general point $x \in X$  an analytic subset of codimension two which is everywhere tangent to $\mathcal F$.
 The existence of a morphism $\pi:X \to S$ to a normal surface and a foliation $\mathcal G$ on $S$ such that $\pi^*\mathcal G = \mathcal F$ follows from standard properties of the Chow's scheme (see \cite[Lemma 2.4]{LPTtrivialcanonical}).

Assume now that we are in case (b). Maybe passing to a finite cover, we can assume
that $N\mathcal F$ is trivial on a Zariski neighborhood of $Y$.
The normal bundle can be expressed by a divisor supported on finitely many fibers of $f$. Therefore, there exists a $1$-form $\beta$ defining $\mathcal F$ with zeros and poles also supported on fibers of $f$. On the one hand Frobenius Theorem implies
that $\beta \wedge d \beta =0$, and on the other hand the closedness of $i_t^*\beta$ for a general $t$ (fibers of $f$ are K\"{a}hler  compact) implies
$dg \wedge d \beta = 0$ for any rational function $g$ constant along the fibers of $f$. Putting these two equations together yields
\[
   d \beta = h dg \wedge \beta  \implies 0 = dh \wedge dg \wedge \beta
\]
for some rational function $h$.
If $dh \wedge dg \neq 0$, then the (irreducible components) of the fibers of $(h,g): X \dashrightarrow \mathbb P^1 \times \mathbb P^1$
are tangent to leaves of $\mathcal F$, and after taking the Stein factorization of $(h,g)$, we obtain a morphism $\pi:X \to S$ to a normal surface and a foliation $\mathcal G$ on $S$ such that $\pi^*\mathcal G = \mathcal F$
as before.
If $dh \wedge dg =0$ then $hdg = f^* \alpha$ for some rational $1$-form on $C$, so we can conclude as in the proof of Proposition \ref{P:infiniteUedatype} that the pull-back of the original foliation by a generically finite morphism is given by a closed rational $1$-form. This clearly implies that the holonomy along $Y$ is virtually abelian and proves the assertion of the theorem when $\utype (Y)>m$.

Now, let us deal with the case $\utype (Y)=m$.  We suppose that the image $G$ of the holonomy representation is not virtually abelian. According to Theorem \ref{THM:Holfactorization}, there exists a morphism $Y\to C$ such that the holonomy representation is finite in restriction to the fibers. Let $F$ be a smooth fiber and $m'$ the order of the holonomy representation restricted to $F$. Note that $m'$ necessarily divides $m$. On some neighborhood of $F$ (in $X$), the foliation is thus defined by a holomorphic first integral $g=y^m$ where $y$ is a locally defined submersive first integral of $\F$ along $Y$. The Ueda connection $(\mathcal U, \nabla)$ is thus trivial along $F$. This easily implies  that the Ueda's class $c\in H^1(Y, {\mathcal O}_Y)$ is trivial along $F$ (this can be done mimicking Ueda's original proof  that $c$ is unambiguously defined up to a constant factor, \cite[\S 2]{Ueda}). On the other hand, recall that $c$ is induced by an element $c'\in H^1(X,{\mathcal O }_X)$ (Remark \ref{r:localglobal}). As $c={c'}_{|Y}$ is not trivial, this means that there exists on $X$ a holomorphic $1$-form $\omega$ (for instance the conjugate of $c'$) such that ${\omega}_{|Y}$  is not identically zero and projects onto $C$. In particular, one can write on a neighborhood $U$ of $F$ in $X$, $\omega=dG$ where $G\in{\mathcal O}( U)$. Intersecting  the levels of $G$ and $g$, one can then fill up a neighborhood of $F$ with codimension $2$ analytic subsets contained in the leaves of $\F$. Let $\Omega\in H^0(X, \Omega_X^1\otimes N\F)$ a twisted one form defining $\F$. From the previous observations, the leaves of the codimension $2$ foliation defined by $\Omega\wedge\omega$ are algebraic and thus provide the sought factorization.\qed

\section{Quasi-smooth foliations}\label{S:QSmoothFol}

In this  section, we will study foliations on projective manifolds having a compact leaf and such that $c_1(N\mathcal F)^2=0$
in $H^4(X,\mathbb C)$. This assumption is certainly satisfied by smooth foliations thanks to Bott's vanishing Theorem. More generally (see for instance \cite{brunellaperrone}),
Baum-Bott index Theorem implies that $N\mathcal F^ 2=0$  for foliations having the following division property: for every local generator $\omega$ of the conormal sheaf ${N\F}^*$ (regarded as an invertible saturated subsheaf of $\Omega_X^1$), there exists some \textit{holomorphic} local one form $\beta$ such that
\[d\omega=\beta\wedge\omega.\]
In particular, this division property holds whenever
every irreducible component $\Sigma$ of  the singular set $\sing(\mathcal F)$ satisfies one of the following conditions:
\begin{enumerate}
\item $\Sigma$ has codimension at least three; or
\item over a  general point of $\Sigma$, $\mathcal F$ admits a holomorphic first integral with critical set contained in $\Sigma$.
\end{enumerate}

\begin{definition}
We will say that a codimension one foliation $\mathcal F$  satisfying $c_1(N\mathcal F)^ 2=0$
is  a quasi-smooth foliation. Furthermore, if the foliation satisfies the division property above then we  will 
say that the foliation is divisible. 
\end{definition}

\subsection{The normal bundle of a quasi-smooth foliation}

\begin{lemma}\label{L:hodge}
If $\mathcal F$ is a quasi-smooth foliation on a compact K\"ahler manifold $X$ admitting a compact leaf $Y$, then
the following assertions hold true.
\begin{enumerate}
\item There exists a rational number $\lambda$ such that $N\mathcal F$ is numerically equivalent to $\lambda Y$.
\item If $r$ is the smallest positive integer such that $r \lambda \in \mathbb Z$ then  $\mathcal L=N\mathcal F^ {\otimes r} \otimes \mathcal O_X(-r\lambda Y)$ is in $ \Pict(X)$ (the group of line bundles with torsion Chern class), and $\mathcal L_{|Y}$ coincides with
$NY^{\otimes r (1- \lambda)}$.
\item If $\lambda=1$, then either the image of the holonomy of $\mathcal{F}$ along $Y$ is abelian, or $Y$ is a fiber of a fibration.
\end{enumerate}
\end{lemma}
\begin{proof}
 (1) Since both $c_1(N\mathcal F)^2$ and $c_1(N\mathcal F)\cdot c_1(\mathcal O_X(Y))$ vanish in $H^4(X,\mathbb C)$, Hodge index Theorem
implies that $N\mathcal F$ is numerically equivalent to $\lambda Y$ for some rational number $\lambda$. The conclusion of (2) follows from item (1)
and the fact that $N\mathcal F_{|Y} = \mathcal O_X(Y)_{|Y} = NY$. To prove (3), we argue according to the order of $\mathcal L=N\mathcal F \otimes \mathcal O_X(-Y)$  in $ \Pict(X)$. We first remark that $\mathcal{L}$ cannot be trivial: the Residue Theorem implies that $\mathcal{F}$ cannot be given by a logarithmic $1$-form whose poles are only on $Y$. If $\ord(\mathcal{L})=2$, then $\mathcal{F}$ is given by a logarithmic $1$-form after a double \'etale cover, which is an isomorphism along any connected component of the pre-image of $Y$ ($\mathcal{L}_Y=\mathcal{O}_Y$), and the holonomy is thus abelian. If $\ord(\mathcal{L})>2$, we can argue as in the proof of Proposition \ref{P:abelianunitary} and conclude that $Y$ is a fiber of a fibration.
\end{proof}
\begin{remark}
Let $\omega$ be a K\"ahler form. By Hodge index theorem (assuming again the existence of a compact leaf), one can notice that $\F$ is quasi-smooth if and only if ${c_1({N\F})}^2\wedge[\omega]^{n-2}\geq 0$ ($n=\dim(X)$), which is {\it a priori} a weaker condition.
\end{remark}

\subsection{Factorization of the foliation}
This paragraph is devoted to the proof of the following result.
\begin{thm}\label{T:factorization-quasi-smooth}
Let $Y$ be a compact leaf of a codimension one quasi-smooth foliation $\mathcal F$ on a compact K\"ahler manifold $X$. If the holonomy of $\mathcal F$ along $Y$ is not abelian, then there exists a morphism $\pi : X \to S$ to a surface $S$ and foliation $\mathcal G$ on $S$ such that
$\mathcal F = \pi^*\mathcal G$.
\end{thm}
\noindent Let us note that it gives a positive answer to Question \ref{Q:factorization} in the case of quasi-smooth foliations. We split the proof according to the order of the normal bundle of $Y$.

\subsubsection{Normal bundle of infinite order}

\begin{prop}\label{P:nottorsioncase}
Let $Y$ be a compact leaf of a codimension one quasi-smooth foliation $\mathcal F$ on a compact K\"ahler manifold $X$.
Assume that $\ord(NY) = \infty$ in $\Pict(Y)$. If the holonomy of $\mathcal F$ along $Y$ is not abelian,
then there exists a morphism $\pi : X \to S$ to a surface $S$, and foliation $\mathcal G$ on $S$ such that
$\mathcal F = \pi^*\mathcal G$.
\end{prop}
\begin{proof}
Let $\rho : \pi_1(Y) \to \Diff$ be the holonomy representation of $\mathcal F$ along $Y$.
Since the image of $\rho$ is not abelian, and its linear part is infinite, we can appeal to Theorem \ref{THM:Holfactorization} (see also Remark \ref{r:infinitecase}): there exists a non-constant  morphism $p: Y \to C$ such that the restriction
of $\rho$ to a general fiber $F$ of $p$ has finite order. Moreover, $NY$ is of the form $p^*N + \tau$ where
$N \in \Pic^0(C)$ and $\tau \in \Pict(Y)$ is torsion. The inclusion $p^*: H^0(C, \Omega^1_C) \to H^0(Y,\Omega^1_Y)$
induces a surjection $p_* : \Alb(Y) \to \Alb(C)$. The morphism $p:Y \to C$ can be seen as the Stein factorization of the
composition of the Albanese morphism $\alb : Y \to \Alb(Y)$ with $p_*$.

Consider  the Zariski closure $G$ of the subgroup generated by $NY$ in $\Pict(Y)$.
Since $\ord(NY) = \infty$, we have that $G$ has positive dimension. Moreover, since some power of $NY$ extends to a line bundle
over $X$ (Lemma \ref{L:hodge}) and the restriction morphism $\Pict(X) \to \Pict(Y)$ has finite kernel (Proposition \ref{P:Albanese}), it follows that $G_X= G\cap \mbox{Image}\{\Pict(X) \to \Pict(Y) \}$
has the same dimension as $G$. In particular $G^{(0)}$, the connected component of the identity of $G$, is contained in $G_X$.

If we dualize the inclusions $G^{(0)} \to \Pic^0(Y)$ and $G^{(0)} \to \Pic^0(X)$ we obtain surjective morphisms to $\Alb(Y) \to A$ and $\Alb(X) \to A$ where $A$ is a compact torus with the following commutative diagram
 $$\xymatrix{
    \Alb(X) \ar[r]  & A \\
    \Alb(Y) \ar[u]\ar[ru] &
  }$$

It follows from the definition of $G^{(0)}$ that $G^{(0)} \subset p^* \Pic^0(C)$ and hence we get
the factorization
\[
Y \to \Alb(Y)  \to \Alb(C) \to A \, .
\]
Since the fibers of $Y \to \Alb(C)$ have codimension one in $Y$, the
same holds true for $\varphi_Y:Y\to A$. Thus both morphisms have the same Stein factorization and the restriction of $\rho$ to fibers of $Y \to A$ is finite.

Now, consider the morphism $\varphi_X: X \to A$.  Since its restriction to $Y$
coincides with $\varphi_Y$, which has codimension one fibers
in $Y$, it follows that the fibers of $\varphi_X$ have codimension one or
two in $X$.

If a general fiber $F$ of $\varphi_X$ has  codimension one, then the
restriction of $\mathcal F$ to $F$ is a codimension one foliation which
 has a compact leaf with finite holonomy. It follows that all the
 leaves of $\mathcal F_{|F}$ are algebraic (or more exactly locally closed, as the ambient manifold is not necessarily algebraic). Thus there exists a codimension two
 foliation $\mathcal H$  on $X$ by algebraic leaves tangent to $\mathcal F$. This
 provides us with the existence of the morphism $\pi:X \to S$ and of the foliation $\mathcal G$ on $S$ such that $\mathcal F = \pi^* \mathcal G$ as in the proof of Theorem \ref{THM:Factorization}.

 If the general fiber of $\varphi_X$ has codimension two, then
 we claim that it is contained in a leaf of $\mathcal F$. Let $F$ be a general fiber
 sufficiently close to a fiber $F_0$ contained in $Y$. Since the holonomy representation is finite in $F_0$, there exists an analytic neighborhood $U$ of $F_0$ such that $\mathcal F_{|U}$ admits a holomorphic first integral. The restriction of this first integral to $F$ must be constant by the maximal principle.  Therefore, the fibers of $\varphi_X$ define a foliation $\mathcal H$
 on $X$ by algebraic leaves which is tangent to $\mathcal F$. The result follows as in the previous case.
\end{proof}
\subsubsection{Torsion normal bundle}
\begin{prop}\label{P:torsioncase}
Let $Y$ be a compact leaf of a codimension one quasi-smooth foliation $\mathcal F$ on a projective manifold $X$.
Assume that $\ord(NY) < \infty$ in $\Pict(Y)$. If the holonomy of $\mathcal F$ along $Y$ is not virtually abelian,
then there exists a morphism $\pi : X \to S$ to a surface $S$ and foliation $\mathcal G$ on $S$ such that
$\mathcal F = \pi^*\mathcal G$.
\end{prop}
\begin{proof}
By item (2) of Lemma \ref{L:hodge}, for some suitable positive integer $m$, there exists in a neighborhood of $Y$ a meromorphic section $\Omega$  of ${{N\F}^*}^{\otimes m}$ (seen as an invertible subsheaf of ${\Omega_X^1}^{\otimes m}$) with poles only on $Y$ and which defines the foliation. On the other hand, according to Theorem \ref{THM:Holfactorization}, the holonomy representation $\rho_\F$ (essentially) factors through a fibration onto a curve $f:Y\to C$. If $F$ denotes a smooth fiber of $F$, the restriction of $\rho_\F$ to $F$ has finite image, and there exists a small analytic neighborhood $U$ (in $X$) of $F$ on which the foliation admits a first integral. This can be expressed as $z^{q}$ where $z$ is some local defining coordinate for $\F$ vanishing on $U\cap Y$ and $q$ is the order of the holonomy group along $F$. Consequently, $\F$ is defined on $U$ by
$$\omega=(\frac{dz}{z^{q+1}})^{\otimes m}\in H^0\left(U, {\Omega_X^1}^{\otimes m} (m(q+1)Y)\right).$$

Let us compare the two meromorphic 1-forms $\Omega$ and $\omega$ . They coincide on $U$ up  to a multiplicative meromorphic function $g$ which can be assumed to be holomorphic (up to replacing $g$ by $1/g$). If this function is constant along the leaves, $\Omega$ can be locally expressed as a power of a {\it closed} meromorphic $1$-form; hence, its restriction to a transversal
to the codimension one foliation will be
invariant by the holonomy group along $Y$, and  this group is thus virtually abelian according to Section \ref{S:interpretation}, contrary to our assumptions.

Suppose now that $g$ is not constant.  For simplicity, we will firstly assume  that $q=1$. Let $k$ be the vanishing order of the $2$-form $dg\wedge dz$ along $Y$. We can see  $\omega_z=dg/z^k$ (for $z\not=0$) and
$$\omega_0= \mbox{Res}_{z=0} \left(\frac{dg\wedge dz}{z ^{k+1}}\right)$$
as an analytic family of 1-forms on the (pieces of) leaves of $\F$ parametrized by $z$. For the leaves $ z=const \neq 0$ , $\omega_z$ is an exact 1-form, thus the same holds true for
$\omega_0$, which is also nontrivial. By leafwise integration, one can then construct a holomorphic function $G$ on $U$ which is non constant on $U\cap Y$ but necessarily constant on $F$ and the nearby fibers (by compactness). Thus, on restriction to nearby leaves of $\F$ in $U$, $G$  will also have compact levels. By standard properties of the Chow's scheme of $X$, we get codimension $2$ analytic subsets tangent to the foliation through a general point of $X$, and it allows us to factorize. The case $q>1$ can be reduced to the preceding one replacing $U$ by some suitable finite \'etale cover.
\end{proof}

In the rest of this section, we investigate Question \ref{Q:solvable} under the assumption that the foliation is quasi-smooth. Except for one situation (see Proposition \ref{P:solvablequasismooth}), we are able to give a positive answer to the latter question.

\subsection{Quasi-smooth foliations with abelian holonomy}
\begin{prop}\label{P:abelianquasismooth}
Let $\mathcal F$ be a quasi-smooth foliation on a projective manifold $X$. Assume that $\mathcal F$ has a compact leaf $Y$ with abelian holonomy. Then, there exists a projective manifold $Z$ and generically finite morphism $\pi : Z \to X$ such that $\pi^* \mathcal F$ is defined by a closed rational $1$-form.
\end{prop}	
\begin{proof}
In view of Propositions \ref{P:infiniteUedatype},  \ref{P:finiteUedatype}, and Corollary \ref{C:finiteUedatype} the only case to deal with is $\utype(Y)=\ord(NY)=\infty$ and formally linearizable holonomy. In this situation, it suffices to prove that the hypothesis of Proposition \ref{P:abelianunitary} are fulfilled Keeping the notations of item (2) of Lemma \ref{L:hodge}, consider $\mathcal L=N\mathcal F^ {\otimes r} \otimes \mathcal O_X(-r\lambda Y)$. If $\lambda=1$, we are done. Assume that $\lambda\not=1$, then  $\mathcal L_{|Y}
=NY^{\otimes r (1- \lambda)}$ is not trivial (recall that $NY$ has infinite order). By Theorem \ref{thm:neemanextension}, there exists an effective divisor $D$ numerically equivalent to $r(1-\lambda)Y$ whose support is disjoint from $Y$ and such that $\mathcal{L}=\mathcal{O}_X(r(1-\lambda)Y-D)$. Thus we get
\[ N\F\buildrel{num}\over{\sim}Y - \frac{1}{r}D\]
and we can then apply Proposition \ref{P:abelianunitary}.
 \end{proof}

\subsection{Quasi-smooth foliations with solvable holonomy}
\begin{prop}\label{P:solvablequasismooth}
Let $\mathcal F$ be a quasi-smooth foliation on a projective manifold $X$. Assume that $\mathcal F$ has a compact leaf $Y$ with solvable holonomy and that
\begin{enumerate}
\item either the order of $NY$ is finite
\item or the Ueda type of $Y$ is finite,
\end{enumerate}
then $\F$ is transversely affine.
\end{prop}
\begin{proof}
This is already covered by Propositions  \ref{solvableholonomy}, \ref{P:solvablefinitetype} and \ref{P:Virtuallyabelian} (without quasi-smoothness assumptions), except when the holonomy group is formally conjugated to the group generated by  \[f(z) = \frac{z}{(1-z^q)^{1/q}} \quad \text{ and } \quad g(z) = \exp(\frac{2\pi i}{2q}) z \] and $Y$ has finite Ueda's type. Suppose now that we are in this latter case. Recall (Section \ref{SS:virtuallyabelian}) that $\F$ is defined in $Y(\infty)$ by a section  $\omega$ of
$$\left({N\mathcal F}_{|Y(\infty)}^* \otimes\mathcal O_{Y(\infty)}((q+1)Y) \right)^{\otimes 2}$$
which can be locally written as ${(dz/z ^{q+1})}^{\otimes 2}$. On the other hand, $\F$ is defined on the whole $X$ by some $\Omega\in H^0(X, {N\F^ *}^{\otimes r} \otimes \mathcal O_X(p Y))$ for some integers $r,p$, $r>0$ (second item of Lemma \ref{L:hodge}). One can suppose that $r=2r'$ is even. We thus obtain that ${\omega}^{\otimes r'}$ and $\Omega$ coincide on $Y(\infty)$ up to a multiplicative factor $F\in {\mathcal O}_{Y(\infty)}$ which is necessarily constant by finiteness of the Ueda type. We conclude observing that $\F$ satisfies the hypothesis of Proposition \ref{P:Virtuallyabelian}.  \end{proof}

Concerning the remaining case, $\utype(Y)=\ord(NY)=\infty$, we have only obtained the following partial result where we use the notion of divisible foliation, notion recalled in the beginning of the present section.
\begin{prop}\label{P:solvabledivisible}
Let $\mathcal F$ be a divisible (hence quasi-smooth)  foliation on a K\" ahler manifold $X$. Assume that $\mathcal F$ has a compact leaf $Y$ with solvable holonomy and that the order of $NY$ and the Ueda type of $Y$ are infinite. Then $\F$ is transversely affine.
\end{prop}
\begin{proof}
Let $\lambda\in\mathbb Q$ such that $N\F$ is numerically equivalent to $\lambda Y$. Thanks to Proposition \ref{P:abelianquasismooth}, one can assume that the holonomy is not abelian. In particular,  $\lambda\neq 1$ (item (3) of Lemma \ref{L:hodge}).

For the sake of simplicity, suppose for the moment that $\lambda$ is an integer. One firstly observe that $\lambda>0$, otherwise the foliation could be defined by a one form $\Omega$  twisted by a flat line bundle $ L$, with zero divisor ${(\Omega)}_0=\lambda Y$. By Hodge identities, we would have $\nabla \Omega=0$ where $\nabla$ is the flat unitary connection attached to $ L$. This clearly implies that the holonomy group along $Y$ is conjugated to a subgroup of $\{z\to az,\ |a|=1\}$, contradicting the non abelianity assumption.

Keeping the notation of Lemma \ref{L:hodge}, consider the flat line bundle $\mathcal L=N\mathcal F \otimes \mathcal O_X(-\lambda Y)$, $\mathcal L_{|Y}=
NY^{ (1- \lambda)}$. According to Theorem \ref{thm:neemanextension}, one can claim that there exists an effective divisor $D$ cohomologous to $(\lambda -1)Y$, $|D|\cap Y=\emptyset$. In particular the line bundle ${\mathcal L}'={\mathcal O} \big ((\lambda -1)Y -D\big ) $ is flat and coincides with ${\mathcal L}^*$ in the neighborhood of $Y$. Actually, these two line bundles coincide on $X$, otherwise the (unitary) monodromy of ${\mathcal L}'\otimes {\mathcal L}$ would be non trivial and we would be able to exhibit (see the proof of Proposition \ref{P:abelianunitary}) on a finite \'etale covering $\tilde X$ of $X$  three pairwise disjoint
effective divisors having numerically trivial normal bundle (both of them being copies of $Y$), implying that $Y$  a fiber of a fibration on $X$ and thus contradicting $\ord (NY)=\infty$. The foliation is thus defined by a global meromorphic form $\Omega$ without zero divisor, and whose polar divisor is equal to $Y+D$. Thus, in order to prove that $\F$ is transversely affine, it is sufficient to exhibit a \textit{global} closed meromorphic one form $\beta$ such that
\begin{equation}\label{e:globalintfact}
\beta\wedge\Omega=d\Omega.
\end{equation}

On the other hand, according to \S 4.4,   $\F$ is defined\footnote{From $\utype (Y)=\infty$, we infer that the linear part of the holonomy is infinite. This implies that this holonomy group is analytically normalizable (Theorem \ref{thm:exceptional}).} in a small connected neighborhood $U$ of $Y$ by a twisted \textit{meromorphic} one form $\Omega_1$ with pole of order $p+1$ along $Y$
$$\Omega_1\in H^0(U,{\Omega}^1((p+1)Y)\otimes E)$$
such that $E$ is a line bundle equipped with a flat holomorphic connection $\nabla$ with respect to which $\Omega_1$ is \textit{closed}. Note that $E$ is identified on $U$ with ${\mathcal O}(-pY)$ and then carries two flat structure: the first one being defined  by $\nabla$ and the second one being the flat unitary structure on ${\mathcal O}(-pY)$, which makes sense because the Ueda connection is trivial on an euclidean neighborhood of  $Y$ (Theorem \ref{thm:neemanextension}).  By $\nabla$-closedness, one can then deduce that there exists a \textit{holomorphic} one form $\eta=\nabla_u-\nabla$ (a priori only defined on $U$) such that
\begin{equation} \label{e:localintfact}
\eta\wedge\Omega_1=\nabla_u\Omega_1.
\end{equation}
where $\nabla_u$ is the flat unitary connection attached to $E$.

Now, observe that if $\omega$ is a meromorphic defining $1$-form for $\F$ at some point  $m$,
if $\omega_1$ is a germ of meromorphic one form at $m$ satisfying $\omega_1\wedge \omega=d\omega$ and if $f$ is a germ of meromorphic function, then
\[(\omega_1+df/f)\wedge (f\omega)=d(f\omega).\]
According to this rule, the equality (\ref{e:localintfact}) can be rewritten as
\[ (\eta-\frac{\nabla_{ u} F}{ F})\wedge\Omega=d\Omega\]
where $\nabla_{u}$ is the unitary flat connection attached to $\mathcal L$ and $F$ is a meromorphic section of ${\mathcal O} (-pY)$  defined on $U$ and  such that $(F)_\infty= pY$.

Set $\beta_U= (\eta-\frac{\nabla_{u} F}{ F})$. This is a closed form (with pole of order one on $Y$) defined a priori only on $U$. We claim that $\beta_U$ extends meromorphically on $X$, thus giving the sought integrating factor $\beta$ in Equation (\ref{e:globalintfact}).

Let us proceed with the proof.  Let $(U_i)$ be a sufficiently fine open covering of $X$. By virtue of the divisibility assumption, we can exhibit a meromorphic form $\beta_i$ on $U_i$ which satisfies $\beta_i\wedge\Omega=d\Omega$. Let us set
\begin{enumerate}
\item $\beta_i={\beta_U}_{|U_i}$ if $U_i\cap Y\neq \emptyset$,
\item $\beta_i= \alpha_i -df_i/f_i$ if $U_i\cap Y=\emptyset$, where $\alpha_i\in\Omega^1 (U_i)$ and $f_i\in{\mathcal O}(U_i)$ is a defining function of $D\cap U_i$. The 1-form $\alpha_i$ is given by the divisibility property for the local (holomorphic) generator $f_i\Omega$ of the foliation $\F$ on $U_i$.
\end{enumerate}
Thanks to this local expression of $\beta_i$, we can notice that $\beta_i-\beta_j\in\Omega^1 (U_i\cap U_j)$. In particular, and because $\Omega$ has no zeros in codimension one, we have
\[\beta_i-\beta_j= h_{ij}\Omega\]
where $h_{ij}\in {\mathcal O}(U_i\cap U_j)$ is a \textit{cocycle}. By construction, this cocycle is trivial in restriction to $Y$, hence has a trivial class in $H^1(X,{\mathcal O }_X)$ thanks to Proposition \ref{P:Albanese}.  Then, there exists on $X$ a meromorphic $1$-form $\omega_1$ satisfying $\omega_1\wedge \Omega=d\Omega$ which can be written on $U_i$ as $\omega_1= \beta_i+h_i\Omega$ where $h_i\in{\mathcal O} (U_i)$. In particular, on $U$ one obtains $\omega_1= \beta+h\Omega$ where $h\in {\mathcal O}(U)$  is indeed constant as $\ord(NY)=\infty$. This yields the closed extension of $\beta$ as wanted.

If $\lambda$ is not an integer, we reduce to the previous case by ramified covering trick.
\end{proof}

\subsection{Quasi-smooth foliations with holonomy tangent to identity}

\begin{prop}\label{tangentidentity}
Let $Y$ be a smooth compact divisor on a projective manifold $X$ such that $\utype(Y)<\infty$. Assume that $Y$ is a compact leaf of a quasi-smooth foliation $\mathcal F$ on $X$ such that the holonomy of $\mathcal F$ along $Y$ is tangent to identity.  Then $\F$ can be defined by a meromorphic closed one form whose polar divisor coincides with $2Y$. In particular, the holonomy of $\F$ along $Y$ is abelian.
\end{prop}
\begin{proof}
Recall that in this context, $NY$ is trivial and $\utype(Y)<\infty$ means exactly $\utype(Y)=1$ by Theorem \ref{thm:Neeman}. Let us consider
$$L= N\F \otimes {\mathcal O}_X(-2Y).$$
All we have to do is to prove that $L$ is trivial on an euclidean neighborhood $U$ of $Y$. Let us take it for granted one moment and see how to conclude the proof. In this case, the foliation $\F$ is defined on $U$ by a meromorphic one form $\Omega$ whose polar locus is $2Y$. Moreover, Theorem \ref{THM:Existence} provides us with the existence of a closed rational $1$-form $\omega$ with poles of order $2$ along $Y$. Up to multiplying $\Omega$ by a suitable scalar number $\lambda\in{\mathbb C}^*$, one can observe that $\omega-\Omega=\omega_0$ is a well defined form in $U$ with at worst a logarithmic pole on $Y$. Assume firstly that $\dim(X)=2$. We can apply Theorem \ref{T:Uedaextension} and extend $\omega_0$ as a holomorphic\footnote{Applying Theorem \ref{T:Uedaextension} only yields a meromorphic extension since we do not know \emph{a priori} that $\omega_0$ is closed. However, extending $\omega$ shows that $N\F$ is linearly equivalent to $2Y+E$ where $E$ is a (non necessarily effective) divisor supported on curves contained in $X\backslash Y$ and that can be contracted (Theorem \ref{T:Ueda}). It follows that the intersection form is negative definite on $E$ and using that $N\F^2=0$, we conclude that $\omega_0$ is holomorphic on $X\backslash Y$.} form on $X\backslash Y$. The Residue Theorem shows that $\omega_0$ is in fact holomorphic on the whole of $X$ and is hence closed. At the end, the foliation $\F$ is given by the (global) closed rational one form $\Omega=\omega+\omega_0$. If $\dim(X)\geq 3$, we can apply the argument above to a surface obtained as a general complete intersection in $X$ and extend the one form to the ambient space using results of Section \ref{sec:extension}.

Now we prove that $L$ is trivial on a neighborhood of $Y$. It is enough to prove that $L$ is flat: if $L$ is given by a locally constant cocycle, this cocycle has to be trivial on a neighborhood of $Y$ since $L_{\vert Y}$ is trivial. By Lemma \ref{L:Uedaholonomy}, we first note that $L_{\vert Y(1)}$ is trivial. Now, let us use Lemma \ref{L:hodge}: there exists $\lambda\in \Q$ such that $N\F$ is linearly equivalent to $\mathcal{O}_X(\lambda Y)$. If $r$ is an integer such that $r\lambda\in\mathbb{Z}$, we can write:
$$\mathcal{O}_X(r(\lambda-2)Y)=L^{\otimes r}\otimes\underbrace{\mathcal{O}_X(r\lambda Y)\otimes N\F^{*\otimes r}}_{\textrm{flat}}.$$
From the triviality of $L_{\vert Y(1)}$, we deduce that $\mathcal{O}_X(r(\lambda-2)Y)_{\vert Y(1)}$ is flat (\emph{i.e.} given by a locally constant cocycle) and then trivial since its restriction to $Y$ is so (the cocycle being locally constant, being trivial on $Y$ is equivalent to being trivial on $Y(1)$). But the assumption $\utype(Y)=1$ implies that no multiple of $\mathcal{O}_X(Y)$ can be trivial when restricted to $Y(1)$ and we conclude that the only possibility is $\lambda=2$ and $L$ is flat. In view of the above argument, it ends the proof.
\end{proof}

\begin{remark}
We cannot expect a result analogous to Proposition \ref{tangentidentity} when $\utype(Y)=\infty$. For instance, take $C_g$ a compact curve of genus $g\geq 2$ and consider the ruled surface $X=C_g\times {\mathbb P}^1$. Let $z$ be a projective coordinate on ${\mathbb P}^1$ and $\omega_1$, $\omega_2\in \Omega^1(C_g)-\{0\}$. The foliation given on $X$ by the rational form
$$dz+z^2\mbox{pr}_1^*\omega_1+z^3\mbox{pr}_1^*\omega_2$$ leaves the fiber $Y:=\{z=0\}$ invariant, is regular in the neighborhood of $Y$ and has holonomy along $Y$ tangent to identity at order one. Let us also remark that $N\F={\mathcal O} (3Y)$ and accordingly that $\F$ is quasi-smooth.  However, this holonomy is not abelian, hence not solvable as soon as  $\omega_1$ and $\omega_2$ are $\C$-independent. Indeed, assume that abelianity holds. Recall that $\omega$ admits a formal integrating factor $g(z)= U/z^2$ i.e, $\displaystyle{d({g\omega})=0}$ where $U$ is a unit in $Y(\infty)$ and consequently depends only on the variable $z$.  On the other hand, the residue of $d(g{\omega})/z$ along $Y$ is equal to $U'(0)\omega_2 + U(0)\omega_3$ whose vanishing implies that $\omega_1$ and $\omega_2$ are $\C$-dependent.
\end{remark}

\begin{remark}
Except for Propositions \ref{P:solvabledivisible} and \ref{tangentidentity}, the propositions established in Section \ref{S:QSmoothFol} remain valid (without any fundamental changes) replacing "quasi-smooth" by
\[ N\F\buildrel{num}\over{\sim}\lambda Y + D\]
where $\lambda$ is a rational number and $D$ is a $\Q$ divisor whose support does not intersect $Y$. Notice that we are not aware of a single example foliation possessing a compact leaf and which does not satisfy this property.
\end{remark}

\appendix

\section{Extension of transverse structures}\label{sec:extension}

Here, we prove an extension result for transverse structures needed in various places to reduce the proofs to the surface case. Precisely, we need to
extend affine transverse structure from a general $2$-dimensional section to the ambient space.
This was proved in \cite{TheseDeFred} in the local setting. The analogous extension result for
meromorphic/rational first integrals, or for Euclidean structure (foliation defined by meromorphic/rational closed
$1$-form) is classical, see \cite{CeMa,CaCeDeBook}.
Here, we provide a proof which works for more general projective structures, which is missing in the literature.
We explain at the end how to adapt to the easier affine case.

In the local/projective setting, a transversely projective foliation $\F$ is the data of a triple $(\omega_0,\omega_1,\omega_2)$
of meromorphic/rational $1$-forms satisfying
\begin{equation}\label{E:FullIntegrability}
\left\{\begin{matrix}
d\omega_0&=&\hfill\omega_0\wedge\omega_1\\
d\omega_1&=&2\omega_0\wedge\omega_2\\
d\omega_2&=&\hfill\omega_1\wedge\omega_2
\end{matrix}\right.
\end{equation}
with $\omega_0\not\equiv 0$. The foliation $\F$ is defined by $\omega_0=0$, and outside of poles of $\omega_i$'s,
we can deduce from the triple a collection of local first integrals for $\F$ that are unique
up to left composition by Moebius transformations. For more details, see \cite{Scardua,croco1,croco5}.

Any two triples $(\omega_0',\omega_1',\omega_2')$ will define the same foliation, with the same
collection of first integrals (up to Moebius transformations) outside of poles, if, and only if,
it can be deduced from the initial triple by a combination of
\begin{equation}\label{E:changeTriple}
\left\{\begin{matrix}
\tilde\omega_0&=&f\cdot\omega_0\\
\tilde\omega_1&=&\omega_1 - \frac{df}{f}\\
\tilde\omega_2&=&{\frac{1}{f}}\cdot\omega_2
\end{matrix}\right.
\ \ \
\left\{\begin{matrix}
\tilde\omega_0&=&\omega_0\hfill\\
\tilde\omega_1&=&\omega_1 + 2g\omega_0\hfill\\
\tilde\omega_2&=&\omega_2+g\omega_1+g^2\omega_0-dg
\end{matrix}\right.
\end{equation}
with $f,g$ meromorphic/rational. Given $\F$, given $\omega_0$ a meromorphic/rational $1$-form
defining $\F$, it is easy to construct another $1$-form $\omega_1$ such that
$d\omega_0=\omega_0\wedge\omega_1$ (see \cite{CeMa,CaCeDeBook,croco1});
it is unique up to addition by a $1$-form proportional to $\omega_0$.
One easily check that any projective triple for $\F$ is equivalent to a triple $(\omega_0,\omega_1,\omega_2)$
with the given $(\omega_0,\omega_1)$; in other words, a projective structure for $\F$ is
equivalent to the data of a $1$-form $\omega_2$ satisfying (\ref{E:FullIntegrability}) with respect
to the given $1$-forms $(\omega_0,\omega_1)$. We will call $(\omega_0,\omega_1)$ a preprojective data for $\F$.

Let $(\underline z,w)\in\mathbb C^{n+1}$ with variables $\underline z=(z_1,\ldots,z_n)$
and consider the hyperplane section $\Sigma=\{w=0\}$.

\begin{lemma}\label{ExtStructLemmaReg}
Let $\F$ be a {\bf regular} codimension one foliation at the origin $p$ of $\C^{n+1}$ and assume $\Sigma$
is not $\F$-invariant. Then, any transversely projective structure for the restriction $\underline{\F}=\F_{\vert\Sigma}$
extends uniquely as a transversely projective structure for $\F$ at the neighborhood of $p$.

More precisely, let $(\underline{\omega}_0,\underline{\omega}_1,\underline{\omega}_2)$ be a projective structure
for the restriction $\underline{\F}$ and let $(\omega_0,\omega_1)$ be a preprojective data for $\F$
extending $(\underline{\omega}_0,\underline{\omega}_1)$ at the neighborhood of $p$.
Then $\underline{\omega}_2$ admits a unique meromorphic extension $\omega_2$ such that
$(\omega_0,\omega_1,\omega_2)$ is a projective triple for $\F$.
\end{lemma}
\begin{proof}At the neighborhood of $p$, let $H$ be a primitive first integral for $\F$, with $H(p)=0$.
The restriction $\underline H:=H_{\vert \Sigma}$ is a primitive integral for $\underline{\F}$
as well (see \cite{MatteiMoussu}). We have $\omega_0=FdH$ and $\omega_1=-\frac{dF}{F}+GdH$
for some meromorphic functions $F,G$ at the neighborhood of $p$.
In restriction to $\Sigma$, we can redefine the projective structure by a unique triple
$(\underline{\omega}_0',\underline{\omega}_1',\underline{\omega}_2')$ with
$(\underline{\omega}_0',\underline{\omega}_1')=(d\underline{H},0)$; from integrability conditions,
we have
$$0=d \underline{\omega}_1'=2\underline{\omega}_0'\wedge\underline{\omega}_2'=2\underline{F}d\underline{H} \wedge\underline{\omega}_2'$$
and therefore $\underline{\omega}_2'=\phi(\underline{H})d\underline{H}$ for some germ of meromorphic function $\phi$
on $(\C,0)$. The extension
$$\omega_2':=\phi(H)dH$$
defines a projective triple for $\F$ at the neighborhood of $p$ extending the given projective structure in restriction to $\Sigma$.
Finally, using the change of projective triple defined by $F$ and $G$, we can deduce the extension $\omega_2$
from $\omega_2'$ at the neighborhood of $p$.
\end{proof}

We note that we have not used that $\Sigma$ was of codimension $1$; it could be even a curve
provided that it is not $\mathcal F$-invariant. For the singular case below, we really need that $\Sigma$
is a dimension $\ge2$ section.

\begin{lemma}\label{ExtStructLemmaSing}
Let $\F$ be a {\bf singular} codimension one foliation at the origin $p$ of $\C^{n+1}$, $n\ge2$. Assume $\Sigma$
is not $\F$-invariant and cutting-out the singular set $S:=\mathrm{Sing}(\F)$ in codimension $2$.
Then, any transversely projective structure for the restriction $\underline{\F}=\F_{\vert\Sigma}$
extends uniquely as a transversely projective structure for $\F$ at the neighborhood of $p$, as in Lemma \ref{ExtStructLemmaReg}.
\end{lemma}
\begin{proof} Consider an open neighborhood $\underline U$ of $p$ in $\Sigma$ together with a Hartogs-like
domain $\underline V\subset \underline U$ whose domain of holomorphy is $\underline U$, but not containing the codimension $2$ set $\underline S:=S\cap\Sigma$.
For instance, one can choose a small polydisc for $\underline U$, and deduce $\underline V$ by deleting the $\epsilon$-neighborhood
of $\underline S$ in $\underline U$ for $\epsilon>0$ small enough.
Fix $(\omega_0,\omega_1)$ on the neighborhood $U$ of $\underline U$ (in the ambient space) such that
$\omega_0$ is defining $\F$ on $U$ and $d\omega_0=\omega_0\wedge\omega_1$. The projective structure
for the restriction $\underline\F$ is defined by $(\underline{\omega}_0,\underline{\omega}_1,\underline{\omega}_2)$
for a (unique) meromorphic $1$-form $\underline{\omega}_2$ on $\underline U$.
By Lemma \ref{ExtStructLemmaReg}, the $1$-form $\underline{\omega}_2$ extends as a meromorphic $1$-form
$\omega_2$ on the neighborhood $V$ of $\underline V$ (in the ambient space) extending the projective structure
for $\F_{\vert V}$. The domain of holomorphy $\underline V$ of $V$ obviously contains a neighborhood of $p$.
Consequently, the meromorphic $1$-form $\omega_2$ extends on  $\underline V$, extending by the way
the projective structure of $\F$ on a neighborhood of $p$.
\end{proof}

A direct consequence of the above lemmata is

\begin{cor}
Let $\F$ be a (singular codimension one) foliation on a projective manifold $X$ of complex dimension $\geq 3$ and
let $\Sigma\subset X$ be a smooth hypersurface. We assume that $\underline S:=\mathrm{Sing}(\F)\cap\Sigma$
has codimension $2$ in $\Sigma$. If $\F_{\vert\Sigma}$ is transversely projective
in restriction to $\Sigma$, then it is also transversely projective
on an Euclidean neighborhood $U$ of $\Sigma$. More precisely, the projective structure in $\Sigma$
extends on $U$ in the following sense: given $(\omega_0,\omega_1)$ on $X$,
the $1$-form $\underline{\omega}_2$ defining the projective structure for $\underline{\F}$ in $\Sigma$
extends uniquely on $U$.

If $\Sigma$ is a general hyperplane section (with very ample normal bundle),
then the projective structure actually extends on the whole of $X$.
\end{cor}

All above results remain valid when replacing projective transverse structure by
affine transverse structure, Euclidean transverse structure or meromorphic/rational first integral.
For instance, given $\omega_0$ defining $\F$, an affine transverse structure is equivalent
to the data of a meromorphic $1$-form $\omega_1$ satisfying $d\omega_0=\omega_0\wedge\omega_1$ and $d\omega_1=0$.
If $\F$ is locally defined by a (minimal) holomorphic first integral $H$, then we can choose $\omega_0=dH$ and
an affine structure is defined by $\omega_1=\phi(H)dH$. It is therefore straightforward to adapt the proof of the previous lemmata
to the affine case.
%\end{appendix}

\bibliographystyle{amsplain}

\end{document}